\documentclass[11pt,letterpaper]{article}
\usepackage{amsmath,amssymb,amsthm}
\usepackage{MnSymbol} 
\usepackage{graphicx,color}
\usepackage[hyphens]{url}
\usepackage{dsfont}
\usepackage{booktabs}
\usepackage[normalem]{ulem}
\usepackage{mathtools}
\usepackage[nameinlink,capitalize]{cleveref}
\usepackage{multirow}
\usepackage{algorithm}
\usepackage[noend]{algpseudocode}
\usepackage{tikz} %

\usepackage{soul} %

\usepackage[square,sort,numbers]{natbib}
\usepackage[sort,nocompress,noadjust]{cite}

\usepackage{subcaption}

\usepackage{tabularx}
\newcolumntype{L}{>{\raggedright\arraybackslash}X}

\numberwithin{equation}{section}
\newtheorem{theorem}{Theorem}[section]
\newtheorem{cor}[theorem]{Corollary}
\newtheorem{lemma}[theorem]{Lemma}
\newtheorem{remark}[theorem]{Remark}
\newtheorem{prop}[theorem]{Proposition}
\newtheorem{obs}[theorem]{Observation}

\newtheorem{defin}[theorem]{Definition}

\newcommand{\cD}{\mathcal{D}}

\newcommand{\cK}{\mathcal{K}}

\newcommand{\cN}{\mathcal{N}}

\newcommand{\cP}{\mathcal{P}}

\newcommand{\R}{\mathbb{R}}

\newcommand{\N}{\mathbb{N}}

\newcommand{\Prob}{\mathbb{P}}

\newcommand*{\Otilde}{\tilde{O}}
\DeclareMathOperator{\logeps}{\log\tfrac{1}{\eps}}

\newcommand*{\eps}{\varepsilon}

\DeclareMathOperator*{\argmin}{argmin}
\DeclareMathOperator*{\argmax}{argmax}
\DeclareMathOperator*{\esssup}{ess\,sup}
\renewcommand{\leq}{\leqslant}
\renewcommand{\geq}{\geqslant}
\DeclareMathOperator{\half}{\frac{1}{2}}

\providecommand{\abs}[1]{\left\lvert#1\right\rvert}

\newcommand{\pieta}{\pi_{\eta}}

\newcommand{\proj}{\Pi_{\cK}}
\newcommand{\projlr}[1]{\proj\left[#1\right]}

\newcommand{\sig}{\sigma}

\newcommand{\Dalp}[2]{\cD_{\alpha}(#1 \; \| \; #2)}
\newcommand{\Dalplr}[2]{\cD_{\alpha}\left(#1 \; \| \; #2\right)}
\newcommand{\Dalpshift}[3]{\cD_{\alpha}^{(#1)}(#2 \; \| \; #3)}
\newcommand{\Dalpshiftlr}[3]{\cD_{\alpha}^{(#1)}\left(#2 \; \| \; #3\right)}

\newcommand{\KLshiftlr}[3]{\mathsf{KL}^{(#1)}\left(#2 \; \| \; #3\right)}

\newcommand{\Hell}{\mathsf{H}}
\newcommand{\TV}{\mathsf{TV}}
\newcommand{\KL}{\mathsf{KL}}
\newcommand{\KLarg}[2]{\mathsf{KL}(#1 \; \| \; #2)}
\newcommand{\KLlrarg}[2]{\mathsf{KL}\left(#1 \; \| \; #2\right)}

\newcommand{\CNI}{\mathrm{CNI}}

\usepackage[margin=1in]{geometry}

\makeatletter
\def\blfootnote{\gdef\@thefnmark{}\@footnotetext}
\makeatother

\begin{document}

	\title{Resolving the Mixing Time of the Langevin Algorithm to its Stationary Distribution for Log-Concave Sampling}
		
	\author{
		Jason M. Altschuler\footnote{This work was done in part while JA was an intern at Apple during Summer 2021. JA was also supported in part by NSF Graduate Research Fellowship 1122374, a TwoSigma PhD Fellowship, and a Faculty Fellowship from the NYU Center for Data Science.}\\
		NYU\\
		\texttt{ja4775@nyu.edu}
		\and
		Kunal Talwar \\
		Apple \\
		\texttt{ktalwar@apple.com}
	}
	\date{}
	\maketitle

\begin{abstract}
	Sampling from a high-dimensional distribution is a fundamental task in statistics, engineering, and the sciences. A canonical approach is the Langevin Algorithm, i.e., the Markov chain for the discretized Langevin Diffusion. This is the sampling analog of Gradient Descent. Despite being studied for several decades in multiple communities, tight mixing bounds for this algorithm remain unresolved even in the seemingly simple setting of log-concave distributions over a bounded domain. This paper completely characterizes the mixing time of the Langevin Algorithm to its stationary distribution in this setting (and others).
	This mixing result can be combined with any bound on the discretization bias in order to sample from the stationary distribution of the continuous Langevin Diffusion. In this way, we disentangle the study of the mixing and bias of the Langevin Algorithm.

	\par Our key insight is to introduce a technique from the differential privacy literature to the sampling literature. This technique, called Privacy Amplification by Iteration, uses as a potential a variant of R\'enyi divergence that is made geometrically aware via Optimal Transport smoothing. This gives a short, simple proof of optimal mixing bounds and has several additional appealing properties. First, our approach removes all unnecessary assumptions required by other sampling analyses.
	Second, our approach unifies many settings: it extends unchanged if the Langevin Algorithm uses projections, stochastic mini-batch gradients, or strongly convex potentials (whereby our mixing time improves exponentially). Third, our approach exploits convexity only through the contractivity of a gradient step---reminiscent of how convexity is used in textbook proofs of Gradient Descent. In this way, we offer a new approach towards further unifying the analyses of optimization and sampling algorithms. 
\end{abstract}

	\newpage
	\setcounter{tocdepth}{2}
	\tableofcontents

	\newpage
	\normalsize

\section{Introduction}\label{sec:intro}

Generating samples from a high-dimensional distribution is a ubiquitous problem with applications in diverse fields such as machine learning and statistics~\citep{andrieu2003introduction,robert1999monte}, numerical integration and scientific computing~\citep{james1980monte,jerrum1996markov,liu2001monte}, and differential privacy~\citep{McSherryT07,DworkRo14}, to name just a few. Here we revisit this problem of sampling in one of the most foundational and well-studied settings: the setting where the target distribution $\pi$ is log-concave, i.e., 
\[
\pi(x) \propto e^{-f(x)}
\]
for some convex potential $f : \R^d \to \R$. The purpose of this paper is to investigate the mixing time of a canonical algorithm---the Langevin Algorithm---for this problem.

\paragraph*{The Langevin Algorithm.} A celebrated, classical fact is that one can generate a sample from $\pi$ via the \emph{Langevin Diffusion}, i.e., the solution to the Stochastic Differential Equation
given by
\begin{align}
	dX_t = -\nabla f(X_t) + \sqrt{2}d B_t\,,
	\label{eq:ld}
\end{align}
where $B_t$ is a standard Brownian motion on $\R^d$. More precisely, under mild conditions on the tail growth of $f$, the distribution of $X_t$ tends to $\pi$ in the limit $t \to \infty$, see e.g.,~\citep{roberts1996exponential,bhattacharya1978criteria}.

\par Although this continuous-time Markov chain is not directly implementable, it leads naturally to an algorithm: run a discretization of the Langevin Diffusion~\eqref{eq:ld} for a sufficiently long time. This idea dates back at least to Parisi's work in 1981~\citep{parisi1981correlation}.
In its simplest form, the resulting algorithm is the discrete-time Markov chain given by
\begin{align}
	X_{t+1} = X_t - \eta \nabla f(X_t) + Z_t\,,
	\label{eq:la-simple}
\end{align}
where the stepsize $\eta > 0$ is the discretization parameter, and $Z_t \sim \cN(0,2\eta I_d)$ are independent Gaussians.
In words, this algorithm is identical to the Langevin Diffusion except that it uses the same gradient for $\eta$ units of time before recomputing.
Note that this algorithm is applicable in the common practical setup where $\pi$ is only known up to its normalizing constant (e.g., as is the case in Bayesian statistics), since this normalization is irrelevant for computing $\nabla f$. 
\par A more general form of this discretized update~\eqref{eq:la-simple} enables handling constrained distributions via projection and large-scale finite-sum potentials $f = \sum_{i=1}^n f_i$ via stochastic gradients. Since all of our results apply without change in this more general setup, we slightly abuse notation to call this the Langevin Algorithm (rather than the Projected Stochastic Langevin Algorithm).

\begin{defin}\label{lem:def-langevin-alg}\label{def:la}
	For a convex set $\cK \subseteq \R^d$, potential $f = \sum_{i=1}^n f_i : \cK \to \R$, batch size $b \leq n$, stepsize $\eta > 0$, and initialization $X_0 \in \cK$, the \emph{Langevin Algorithm} is the Markov chain
	\begin{align}
		X_{t+1} = \proj\left[ X_t - \eta G_t + Z_t \right],
		\label{eq:la}
	\end{align}
	where $\proj$ denotes the Euclidean projection onto $\cK$, $Z_t \sim \cN(0,2\eta I_d)$ are independent Gaussians, and $G_t = \frac{1}{b} \sum_{i \in B_t} \nabla f_i(X_t)$ are the average gradients on uniform random batches $B_t \subseteq [n]$ of size $b$.
\end{defin}

This Langevin Algorithm has an immense literature in part because it has been studied in several communities, in either identical or similar forms. For example, this algorithm is the direct analog of (Projected Stochastic) Gradient Descent in convex optimization, the only difference being the noise $Z_t$. As another example, this algorithm (with different noise scaling) has been studied extensively in the differential privacy literature under the names (Projected Stochastic) Noisy Gradient Descent and Differentially Private Gradient Descent. Note that in the sampling, scientific computing, and PDE literatures, this algorithm goes by various other names, such as (Projected Stochastic variants of) the Langevin Monte Carlo, the Unadjusted Langevin Algorithm, the Overdamped Langevin Algorithm, and the Forward Euler or Euler-Maruyama discretization of the Langevin Diffusion.

\begin{figure}
	\centering
	\includegraphics[width=0.55\linewidth]{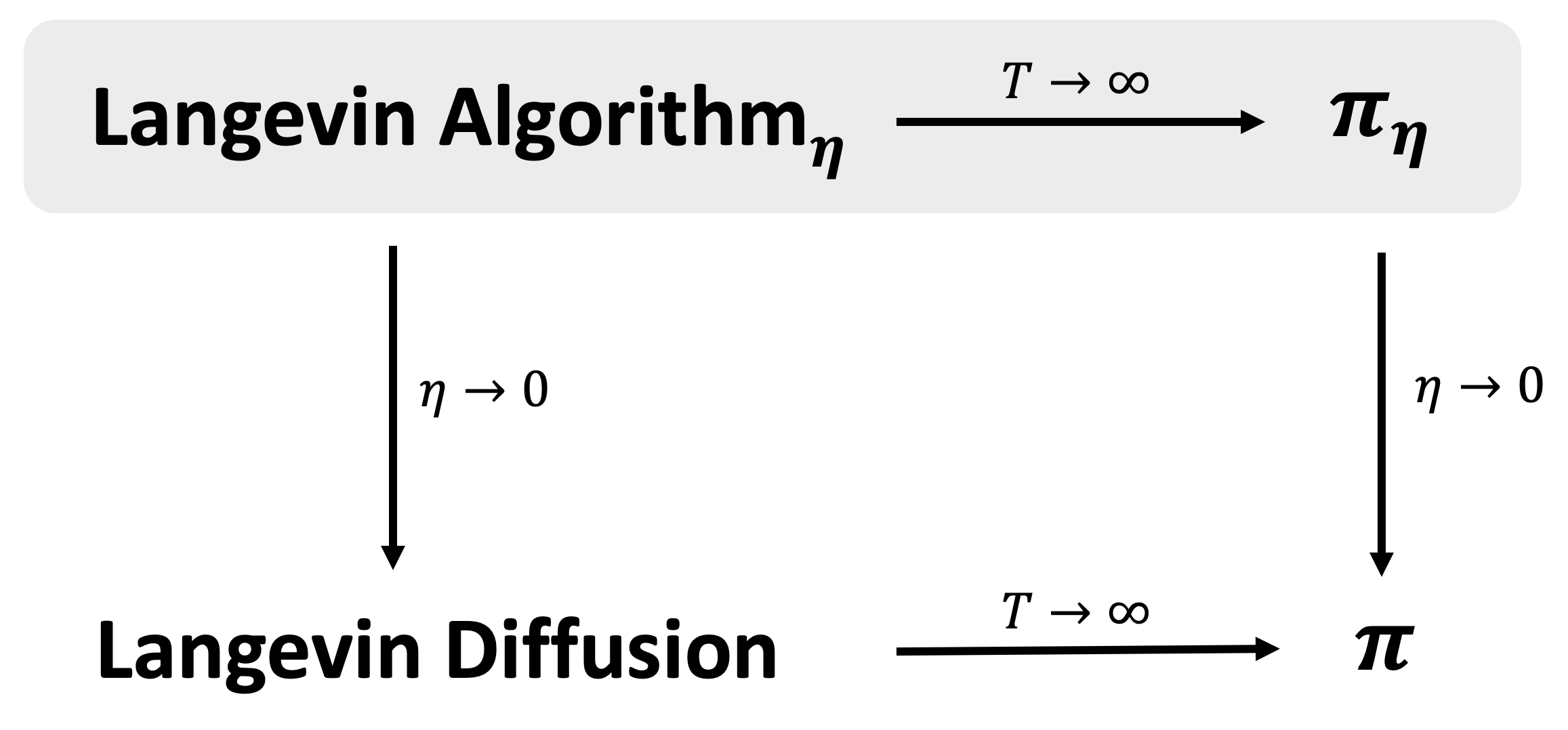}
	\caption{ \footnotesize \textbf{Bottom:} The Langevin Diffusion is a continuous-time Markov chain with stationary distribution $\pi$. \textbf{Top:} The Langevin Algorithm is a discrete-time Markov chain with stationary distribution $\pieta$. This paper establishes tight bounds for the Langevin Algorithm to mix to its stationary distribution $\pieta$ (the boxed arrow) for any non-trivial discretization stepsize $\eta > 0$ in the setting where $\pi$ is (strongly) log-concave. If desired, this result can be combined with any bound on the discretization bias in order to sample from $\pi$. In this way, we aim to disentangle the study of the mixing and bias of the Langevin Algorithm.}
	\label{fig:langevindiagram}
\end{figure}

\paragraph*{Rapid mixing of the Langevin Algorithm?} 
While the asymptotic convergence of the Langevin Algorithm is classically understood~\citep{roberts1996exponential,bhattacharya1978criteria},
non-asymptotic convergence bounds (a.k.a., mixing time bounds) are significantly more challenging. Yet this question is essential since algorithms are implemented in non-asymptotic regimes, and moreover since it is in general impossible to obtain an a posteriori error estimate from a sample.

\par There are two types of ``mixing'' results that one might hope to show: convergence of the Langevin Algorithm to its stationary distribution $\pieta$, or to the target distribution $\pi$. Note that these two distributions are different for any discretization stepsize $\eta > 0$; the error is called the bias of the Langevin Algorithm, see Figure~\ref{fig:langevindiagram}. Henceforth, we refer to the convergence of the Langevin Algorithm to its stationary distribution $\pieta$ as ``mixing'' since this is the standard notion of mixing for a Markov chain, and we disambiguate convergence to $\pi$ as ``mixing to $\pi$''. (Note that much of the literature calls the latter notion ``mixing'' even though this is a slight abuse of notation.)

\par There has been an incredible flurry of activity around the mixing time of the Langevin Algorithm since the the first results in the seminal paper \citep{Dalalyan16}. However, tight mixing bounds
(to either $\pieta$ or $\pi$)
remain open even in the seemingly simple setting of strongly convex and smooth potentials over $\R^d$---let alone in less idealized settings where the potential is convex rather than strongly convex, and/or the domain is a subset of $\R^d$ rather than the whole space, and/or the gradients used by the Langevin Algorithm are stochastic rather than exact. See the prior work section \S\ref{ssec:intro-prev} for details.

\par The purpose of this paper is to determine the mixing time of the Langevin Algorithm to its stationary distribution $\pieta$ in all these settings. If desired, these mixing bounds can then be combined with any bound on the discretization bias in order to sample from $\pi$, see Figure~\ref{fig:langevindiagram}. This differs from most previous work which has primarily focused on directly analyzing mixing bounds to $\pi$. In this way, on one hand we answer a natural question about the Markov chain defining the Langevin Algorithm, and on the other hand we aim to disentangle the study of its mixing and bias. See the discussion in \S\ref{sec:discussion} for future directions about understanding the bias.

The rest of the introduction is organized as follows: we detail our contributions in \S\ref{ssec:intro-cont}, our analysis techniques in \S\ref{ssec:intro-tech}, and then further contextualize with related work in \S\ref{ssec:intro-prev}.

\subsection{Contributions}\label{ssec:intro-cont}

The main result of this paper is a characterization of the mixing time of the (Projected Stochastic) Langevin Algorithm to its stationary distribution $\pieta$ for discretization parameter $\eta > 0$ in the log-concave setting.
For simplicity, we state this result below for mixing in the total variation metric. Extensions to other notions of distance are straightforward and summarized in
Table~\ref{tab:mixing}.
In what follows, we denote the mixing time for a divergence or metric $D$ by $T_{\mathrm{mix}, \; D}(\eps)$; this is the smallest $T \in \N$ for which $D(X_T \| \pieta) \leq \eps$ for any initialization $X_0$ of the Langevin Algorithm.

\begin{theorem}[Informal versions of Theorems~\ref{thm:convex-ub} and~\ref{thm:convex-lb}]\label{thm:intro-convex}
	For any convex set $\cK \subset \R^d$ of diameter $D$, any $M$-smooth convex potentials $f_1,\dots,f_n : \cK \to \R$, any batch size $b \leq n$, and any stepsize $\eta \leq 2/M$, the Langevin Algorithm in Definition~\ref{def:la} mixes to its stationary distribution $\pieta$ in time
	\[
	T_{\mathrm{mix}, \,\TV}\left( \frac{1}{4} \right)
	\asymp 
	\frac{D^2}{\eta}\,.
	\]
\end{theorem}

This result is proved by establishing an upper bound (Theorem~\ref{thm:convex-ub}) and lower bound (Theorem~\ref{thm:convex-lb}) that match up to a constant factor.

\paragraph*{Discussion of dependence on parameters.} 

\begin{itemize}
	\item \underline{Dependence on error $\eps$.} As is standard, Theorem~\ref{thm:intro-convex} is stated for mixing to constant error, here $\eps = 1/4$. A basic fact about Markov chains is that a mixing bound to constant error can be boosted to arbitrarily small error $\eps > 0$ at an exponential rate, see, e.g.,~\citep[\S 4.5]{peres2017}. Specifically, Theorem~\ref{thm:convex-ub} implies $T_{\mathrm{mix}, \,\TV}\left( \eps \right) \leq  \frac{D^2}{\eta} \log \frac{1}{\eps}$.
	\item \underline{Dependence on diameter $D$ and stepsize $\eta$.} Note that $\Theta(D^2/\eta)$ is the number of iterations it takes for a univariate random walk with i.i.d.\;$\cN(0,2\eta)$ Gaussian increments to move a distance $D$ from its initialization. Thus, $\Theta(D^2/\eta)$ is the number of iterations it takes for the Langevin Algorithm to just reach the opposite side of a univariate interval $\cK = [-D/2,D/2]$ even in the simple setting of zero potentials (a.k.a., uniform $\pi$). Since reachability is obviously a pre-requisite for mixing, $\Theta(D^2/\eta)$ is therefore a lower bound on the mixing time. (This is the intuition behind the lower bound in Theorem~\ref{thm:convex-lb}.) The matching upper bound in Theorem~\ref{thm:convex-ub} shows that, modulo a constant factor, the Langevin Algorithm takes the same amount of time to \emph{mix} as it does to just \emph{reach} the opposite end of $\cK$.
	\item \underline{Dependence on other parameters.} 
	Note that the fast mixing in Theorem~\ref{thm:intro-convex} depends on no other parameters, and in particular has no dimension dependence for $\eta$ of constant size (a common regime in differentially private optimization). 
	However, if one seeks to sample from $\pi$, then the discretization bias $\pieta \approx \pi$ must be small, which requires $1/\eta$ to scale polynomially in the dimension and inverse accuracy. This is why mixing bounds to $\pi$ inevitably have polynomial dependence on the dimension and inverse accuracy. Theorem~\ref{thm:intro-convex} clarifies that these polynomial dependencies arise solely through the single factor of $1/\eta$.
\end{itemize}

\paragraph*{Streamlined analysis and implications.} We prove Theorem~\ref{thm:convex-ub} by introducing a technique from the differential privacy literature to the sampling literature. See \S\ref{ssec:intro-tech} below for an overview of this technique, called Privacy Amplification by Iteration.
This approach gives a short, simple proof of optimal mixing bounds---and moreover has several additional appealing properties, which are a mix of technical and pedagogical in nature:

\setlength{\tabcolsep}{9pt}
\renewcommand{\arraystretch}{1.5}
\begin{table}[]
	\centering
	\begin{tabular}{c|cc}
		& Log-concave & Strongly log-concave \\ \hline
		$\TV$, $\KL$, $\chi^2$, $\Hell$
		& $\frac{D^2}{\eta} \logeps$
		& $\frac{1}{\eta m} \log \frac{D}{\eta \eps}$
		\\
		$\cD_{\alpha}$
		&
		$\frac{D^2}{\eta} (\alpha + \logeps) $
		& $\frac{1}{\eta m} \log \frac{\alpha D}{\eta \eps}$
	\end{tabular}
	\caption{\footnotesize Summary of our mixing time bounds for the (Projected Stochastic) Langevin Algorithm to its stationary distribution $\pieta$. Here, $\eta$ is the stepsize, $\eps$ is the mixing error, $D$ is the diameter of the convex set $\cK$, $m$ is the strong-convexity constant of the potential $f$, $M$ is the smoothness constant of $f$, and $\alpha$ is the R\'enyi parameter. We take $\eta \leq 1/M$ here to simplify the asymptotics.
	 These results apply to stochastic gradients with arbitrary batch size. Extensions to unconstrained settings are straightforward, see Appendix~\ref{app:diam}. We report mixing times for Total Variation ($\TV$), Kullback-Leibler ($\KL$), Chi-Squared ($\chi^2$), Hellinger ($\Hell$), and R\'enyi divergence ($\cD_{\alpha}$). For full details, see Corollary~\ref{cor:convex-mixing} and Theorem~\ref{thm:sc-ub}
	 for the cases of log-concave and strongly log-concave targets $\pi$, respectively.
	}
	\label{tab:mixing}
\end{table}

\begin{itemize}
	\item \underline{Minimal assumptions.} 
	Theorem~\ref{thm:convex-ub} requires an essentially minimal set of assumptions.
	This is in part because of our direct analysis approach, and also in part because our analysis distenagles the study of the mixing and the bias of the Langevin Algorithm. Indeed, assumptions like curvature of the set and Lipschitzness of the potential\footnote{In the constrained setting, Lipschitzness is implied by smoothness; however, in the unconstrained setting, Lipschitzness is not implied and yet our techniques still apply despite this.} appear necessary for discretization bounds (see, e.g,.~\citep{bubeck2018sampling})---but as we show, these are unnecessary for fast mixing. Note also that our result requires no warm start.
	\par In fact, Theorem~\ref{thm:convex-ub} only makes three assumptions. The first is that $\cK$ is bounded; this is easily relaxed, see the following paragraph. The second is that the potential $f$ and the set $\cK$ are convex; it is unavoidable to have at least some type of assumption in this direction since arbitrary non-convexity leads to computational hardness due to reductions to non-convex optimization (see \S\ref{sec:discussion}). The third is that the stepsize $\eta \leq 2/M$ is not too large, else $\pieta$ is a meaningless approximation  of $\pi$ and in fact is even transient in the unconstrained setting.\footnote{This can be seen even in $1$ dimension with the quadratic potential $f(x) = \tfrac{M}{2}x^2$ on $\cK = [-D/2,D/2]$. Briefly, if $\eta = (2+\eps)/M$, then $\pieta$ concentrates away from $0$, whereas $\pi$ concentrates around $0$, in fact arbitrarily so as $M \to \infty$. Thus $\pi$ and $\pieta$ can be arbitrarily different if $\eta > 2/M$. Moreover in the unconstrained setting, this Markov chain is transient.}
	\item \underline{Unification of different problem settings.} Our analysis technique extends nearly unchanged and yields tight bounds regardless of whether the Langevin Algorithm uses mini-batch stochastic gradients, projections, or strongly log-concave targets $\pi$. 
	In the stochastic setting, our mixing result is identical 
	 regardless of the batch size. In the unconstrained setting, our mixing result is identical except with $D$ replaced by the diameter of a ball that captures all but $\eps$ mass of $\pieta$; such a diameter proxy is finite and leads to tight mixing bounds, details in Appendix~\ref{app:diam}. And in the strongly log-concave setting, our mixing result improves exponentially in $D$; details in \S\ref{sec:sc}.
	\item \underline{Pedagogical hope towards unifying sampling and optimization.} While the problems of optimization and sampling are well-known to exhibit appealing syntactic similarities in terms of problem formulations and algorithms, it is less clear to how these correspondences should extend to analyses and error bounds. One particularly popular approach, 
		initiated by the seminal work~\citep{VempalaW19}, exploits the JKO Theorem which states that the Langevin Diffusion is equivalent to a KL gradient flow over the space $\cP(\R^d)$ of probability distributions~\citep{jordan1998variational}. 
		See the previous work section \S\ref{ssec:intro-prev} for details. The present paper offers a different approach towards unifying the analyses and error bounds of sampling and optimization algorithms that does not require lifting from $\R^d$ to $\cP(\R^d)$, thereby avoiding the (beautiful yet) high-powered associated mathematical machinery. 
		From an analysis perspective, we exploit convexity through the contractivity of a gradient update (in a different way from~\citep{Dalalyan16,Dalalyan2017further}). This is reminiscent of how convexity is used in textbook proofs of Gradient Descent (see, e.g.,~\citep[pages 271 and 279]{bubeck2015convex} respectively for the convex and strongly convex settings),
		and leads us to error bounds for sampling that in many ways mirror classical results in optimization, see Table~\ref{tab:sampling-vs-opt}.\footnote{In fact, it is worth mentioning that our techniques yield analyses and error bounds for sampling which are even simpler than those for optimization in the setting of stochastic gradients---since our sampling bounds are unchanged for any minibatch size $b$, whereas optimization bounds certainly are not.} 
\end{itemize}

\renewcommand{\arraystretch}{1.5}
\begin{table}[]
	\centering
	\begin{tabular}{c|cc}
		& \textbf{Optimization} & \textbf{Sampling}
		\\ \hline
		Task
		&
		Optimize $\min_{x} f(x)$
		&
		Sample $\pi(x) \propto e^{-f(x)}$
		\\
		Simple first-order algorithm
		&
		Gradient Descent
		& 
		Langevin Algorithm
		\\
		Algorithm update
		&
		$X_{t+1} = \proj [X_t - \eta \nabla f(X_t)]$
		&
		$X_{t+1} = \proj [X_t - \eta \nabla f(X_t) + \cN(0,2\eta I_d)]$
		\\
		``Error'' for $f$ convex
		& $MD^2/T$
		& $e^{-T/MD^2} $ (Theorem~\ref{thm:convex-ub})
		\\
		``Error'' for $f$ strongly convex
		& $ c^T MD^2 $
		& $ c^T MD^2 $ (Theorem~\ref{thm:sc-ub})
		\\
	\end{tabular}
	\caption{\footnotesize The research areas of optimization and sampling exhibit appealing syntactic similarities in terms of problem formulations and algorithms. This paper analyzes the mixing of the Langevin Algorithm by exploiting convexity of the potential $f$ (a.k.a., log-concavity of $\pi$) in a way that is reminiscent of the classical analysis of Gradient Descent. This leads to syntactically similar error bounds. 
		To highlight similarities, here the stepsize $\eta \asymp 1/M$, the optimization ``error'' is reported here for function suboptimality $f(X_T) - f^*$~\citep[Theorems 3.3 and 3.12]{bubeck2015convex}, and the sampling ``error'' is reported for mixing to $\pieta$. 
	Here, $T$ is the number of iterations, $D$ is the diameter, $m$ is the strong convexity of $f$, $M$ is the smoothness of $f$, and $c = \max_{\lambda \in \{m,M\}} |1 - \eta \lambda| \leq 1$ is the contraction rate for a Gradient Descent step.
	}
	\label{tab:sampling-vs-opt}
\end{table}

\subsection{Techniques}\label{ssec:intro-tech}

\subsubsection{Upper bound on mixing time}

We establish rapid mixing by introducing to the sampling community a technique from differential privacy. This technique, called Privacy Amplification by Iteration (PABI), was originally proposed in~\citep{pabi} and then recently developed in~\citep{AltTal22dp} into the form used here which is more amenable to sampling. See the preliminaries \S\ref{ssec:prelim:pabi} for a discussion of these differences and for full technical details about PABI. Here we briefly overview this PABI technique and how we use it---in an effort both to promote accessibility to the sampling community, as well as to emphasize the simplicity of this mixing time argument.

\paragraph*{The Langevin Algorithm as a Contractive Noisy Process.} Rather than view the Langevin Algorithm as a discretized approximation of a continuous SDE, we analyze it directly by exploiting the fact that, modulo a slight shifting of the process described in \S\ref{ssec:convex-ub}, the update equation~\eqref{eq:la} of the Langevin Algorithm decomposes into two simple operations:
\begin{itemize}
	\item \emph{A contractive map}, namely the composition of a projection operation and a Stochastic Gradient Descent update, both of which are contractive.
	\item \emph{A noise convolution}, namely add Gaussian noise $\cN(0,2\eta I_d)$.
\end{itemize}
A simple but key insight is that it is harder to distinguish two random variables $X,X'$ from each other after applying either of these operations.
Indeed, applying the same contraction map to two points can only bring them closer together; and adding the same noise to two random variables can only lower the signal-to-noise ratio. Such observations naturally lend themselves to mixing-time analyses, where by definition, one seeks to understand how distinguishable $X_T,X_T'$ are if they are the final iterates of the same Markov chain run on different initializations $X_0,X_0'$. 
\par The key issue is how to quantify this intuition: how much less distinguishable do two random variables become after the (repeated, interleaved) application of a contractive map or noise convolution? This is precisely the sort of question that the PABI technique is designed for.

\paragraph*{Lyapunov function.} 
The PABI technique measures progress through the \emph{shifted R\'enyi divergence}, which is a variant of the classical R\'enyi divergence that is smoothed by the Optimal Transport (a.k.a., Wasserstein) distance\footnote{Recall that the infinity-Wasserstein distance $W_{\infty}(\mu,\mu')$ between distributions $\mu,\mu'$ on $\R^d$ is the smallest $z \geq 0$ for which there exists a joint distribution $P$ with marginals $X \sim \mu$, $X' \sim \mu'$ satisfying $\|X-X'\| \leq z$ almost surely. Throughout this paper, $\|\cdot\|$ denotes the Euclidean norm.}.
For simplicity, here we describe PABI in the special case of shifted KL divergence; see \S\ref{ssec:prelim:pabi} for full details on the shifted R\'enyi divergence, which generalizes the shifted KL divergence by bounding ``all moments'' simultaneously instead of just ``one moment''.

\par The shifted KL divergence is defined as
\[
	\KLshiftlr{z}{\mu}{\nu} := \inf_{\mu' \; : \; W_{\infty}(\mu,\mu') \leq z} \KLlrarg{\mu'}{\nu}\,.
\] 
In words, the shift $z \geq 0$ makes the KL divergence ``geometrically aware''. Indeed, a major drawback of the KL divergence is that it is invariant under a permutation of points in the underlying space $\R^d$. The shifted KL divergence is more natural for analyzing sampling algorithms since closeness of points is captured by the Optimal Transport shift. As a toy illustrative example, consider two Dirac distributions at points $x$ and $x + \eps$, where $\eps$ is a vector of small norm---while the KL divergence between these two Diracs is infinite, their shifted KL divergence is $0$ for any shift $z \geq \|\eps\|$.

\paragraph*{Tracking the Lypaunov function.} Operationally, the PABI technique  is based on two core lemmas, which disentangle how the two aforementioned operations (contraction and convolution) affect the distinguishability of the iterates, as measured in shifted KL divergence. We recall both lemmas here in a simplified form that suffices for the present purpose of mixing analyses; see \S\ref{ssec:prelim:pabi} for a discussion of the differences to PABI in differential privacy analyses.

\begin{itemize}
	\item \emph{The contraction-reduction lemma} establishes that for any contraction $\phi$,
	\[
	\KLshiftlr{z}{\phi_{\#} \mu}{\phi_{\#} \nu}
	\leq
	\KLshiftlr{z}{\mu}{\nu}\,.
	\]
	Intuitively, this can be viewed as a generalization of the classical data-processing inequality for the KL divergence to the shifted KL divergence---so long as the data-processing is contractive.
	\item \emph{The shift-reduction lemma} establishes that for any shift $a \geq 0$,
	\[
	\KLshiftlr{z}{\mu \ast \cN(0,\sig^2 I_d)}{ \nu \ast \cN(0,\sig^2 I_d)}
	\leq
	\KLshiftlr{z+a}{\mu}{\nu}
	+
	\frac{a^2}{2\sig^2}\,.
	\]
	Intuitively, this can be thought of as a strengthening of the trivial bound $	\KL(\mu \ast \xi \| \nu \ast \xi) 
	\leq
	\KL(\mu \otimes \xi \| \nu \otimes \xi)
	= 
	\KL(\mu \| \nu)+ \KL(\xi \| \xi)$
	which follows from the data-processing and tensorization properties of the KL divergence. In words, the shift-reduction lemma sharpens this trivial bound by moving $a$ units of ``displacement'' between $\mu$ and $\nu$ from the first term to the second term---lowering the KL divergence between $\mu$ and $\nu$, at a penalty given by how well the noise $\xi = \cN(0,\sig^2I_d)$ masks a displacement of size $a$. Indeed, this is precisely how the penalty arises, since $\sup_{w : \|w\| \leq a} \KLarg{\cN(w,\sig^2 I_d)}{\cN(0,\sig^2 I_d)} = a^2/(2\sig^2)$. Operationally, this increase in shift allows us to ``hide'' the gap between two different Langevin initializations, as described next.
\end{itemize}

Combining these two lemmas bounds the joint effect of the two operations---a.k.a., the effect of one iteration of the Langevin Algorithm:
\[
\KLshiftlr{z}{(\phi_{\#} \mu) \ast \cN(0,2\eta I_d)}{(\phi_{\#} \nu) \ast \cN(0,2\eta I_d)}
\leq
\KLshiftlr{z + a}{\mu}{\nu} + \frac{a^2}{4\eta}\,,
\]
for any choice of $a \geq 0$.
In particular, we can analyze $T$ iterations of the Langevin Algorithm by repeating this argument $T$ times and setting $a = D/T$, yielding
\[
\KLlrarg{X_T}{X_T'}
=
\KLshiftlr{0}{X_T}{X_T'}
\leq
\KLshiftlr{D}{X_0}{X_0'}
+
T\frac{a^2}{4\eta}
= 
\frac{D^2}{4\eta T}\,.
\]
Above, the first equality is because $0$-shifted KL is the standard KL divergence; and the final equality is because the $D$-shifted KL vanishes if the two distributions are supported on a set of diameter $D$. For constant $\eps$, this is already enough to conclude the claimed $\KL$ mixing time bound of $T_{\mathrm{mix},\KL}(\eps) \lesssim D^2/\eta$. For sub-constant $\eps$, a standard boosting argument improves the dependence on $\eps$ from $1/\eps$ to $\log 1/\eps$.

\paragraph*{Extensions.} This argument readily extends to different settings. For example, it applies regardless of stochastic gradients because, for any choice of minibatch, a Stochastic Gradient Descent update is contractive. As a second example, this fast mixing bound for KL immediately implies a fast mixing bound for other notions of distance (e.g., Total Variation) via standard comparison inequalities (e.g., Pinsker's inequality) and the fact that the mixing contraction rate is slowest for Total Variation among all $f$-divergences; details in \S\ref{app:otherdists}. As a third example, this argument recovers tight, exponentially better mixing rates if the potentials are strongly convex: in this case, the contraction map is $c$-contractive for some $c < 1$, which improves the contraction-reduction lemma by a factor of $c$ and the resulting $T$-iteration argument by a factor of $c^T$; details in \S\ref{ssec:prelim:pabi}.

\subsubsection{Lower bound on mixing time} Our (algorithm-dependent) lower bounds for sampling are loosely motivated by classical lower bounds from convex optimization in the sense that although we use somewhat different constructions and analyses, the key reasons behind the lower bounds are mirrored: the log-concave lower bound is due to \emph{reachability}, and the strongly log-concave lower bound is due to \emph{strong contractivity}. Below we briefly describe the core ideas behind both. As an aside, an appealing property of both lower bounds is that they apply even in the simple setting of univariate quadratic potentials with full-batch gradients.

\paragraph*{Log-concave lower bound via reachability.}
Briefly, the simple but key observation is that if the Langevin Algorithm is initialized near the bottom of the interval $\cK = [-D/2,D/2] \subset \R$ and the potential $f \equiv 0$ is identically zero, then the iterates form a constrained random walk with $\cN(0,2\eta)$ increments---and thus, with high probability, the Langevin Algorithm requires the minimax-optimal number $\Theta(D^2/\eta)$ of iterations before it can even reach the top of the interval $\cK$, let alone mix. Note that this construction is different from optimization lower bounds (since minimizing the identically zero function is trivial). However, this construction intuitively leads to similar rates as in convex optimization due to the same core challenge of reachability---reachability to a minimizer in optimization, and reachability to the full set in sampling; details in \S\ref{ssec:convex-lb}.

\paragraph*{Strongly log-concave lower bound via strong contractivity.}
Here, we can re-use the classical lower bound construction for Gradient Descent on strongly convex functions.  Specifically, consider running the Langevin Algorithm on the univariate quadratic $x \mapsto \frac{\lambda}{2}x^2$ whose second derivative $\lambda \in [m,M]$ is chosen so as to worsen the corresponding contraction coefficient $c = \abs{1 - \lambda \eta}$ for a Gradient Descent update with stepsize $\eta$. In convex optimization, that is already sufficient to prove that Gradient Descent converges at a rate no faster than $c^{\Theta(T)}$; however, for sampling one must also understand how the Gaussian noise affects the probabilistic goal of being close to $\pieta$. Briefly, the simple but key observation is that in the unconstrained setting (which affects the mixing time only up to a logarithmic factor), the Langevin Algorithm's iterates have explicitly computable Gaussian laws since each iteration contracts the previous iterate by a factor of $c$ and then injects fresh Gaussian noise. More precisely, the $T$-th iterate $X_T$ has variance which decays in $T$ at a rate of $1-c^{2T}$, and thus $\KLarg{X_T}{\pieta} = \KLarg{X_T}{X_{\infty}} = \KLarg{\cN(0,1-c^{2T})}{\cN(0,1)}$ which by a direct calculation is of the desired order $ c^{\Theta(T)}$;
details in \S\ref{ssec:sc-lb}.

\subsection{Related Work}\label{ssec:intro-prev}

This paper is at the intersection of several lines of related work. We further contextualize this here.

\paragraph*{Other sampling algorithms.} Due to the many applications of sampling, there is an extensive literature on different sampling algorithms spanning multiple communities and decades. Much of this work was inspired by the seminal paper~\citep{dyer1991random}, which gave the first randomized polynomial-time algorithm for computing the volume of a convex body via a reduction to (log-concave) sampling. A beautiful line of work from the theoretical computer science community has improved the (originally large) polynomial runtime factors by developing elegant algorithms based on geometric random walks such as the Metropolized Random Walk~\citep{mengersen1996rates,roberts1996geometric}, the Ball Walk~\citep{lovasz1990mixing,lovasz1993random}, and the Hit-and-Run Walk~\citep{lovasz2007geometry,lovasz2006hit,lovasz1999hit, kannan1995isoperimetric,belisle1993hit}, to name a few; for details see the survey~\citep{vempala2005geometric}. A key distinguishing aspect of this line of work is that these geometric random walk algorithms only access the target distribution $\pi$ by evaluating the density up to a constant factor---in contrast to the Langevin Algorithm which also evaluates gradient information about the density. This distinction is formalized through the notion of zero-th order and first-order algorithms: the former access $\pi$ only through evaluations of the potential $f$, whereas the latter also evaluate gradients of $f$. Most first-order algorithms are variants, at least intuitively, of the Langevin Algorithm. These variations come in many forms and can be essential for faster runtimes in theory, practice, or both. A few examples of these variations include: dealing with constrained sets by using projected~\citep{bubeck2018sampling}, proximal~\citep{salim2020primal,brosse2017sampling,durmus2018efficient}, or mirror~\citep{ahn2021efficient, zhang2020wasserstein,hsieh2018mirrored} variants of the Langevin Algorithm; augmenting the state space by using Hamiltonian Monte Carlo~\citep{duane1987hybrid,neal2011mcmc} or its variants such as Riemannian Hamiltonian Monte Carlo~\citep{girolami2011riemann,lee2018convergence,kook2022sampling}, NUTS~\citep{hoffman2014no}, or the Underdamped Langevin Algorithm~\citep{cheng2018underdamped,eberle2019couplings,cao2020complexity}; or using different discretization schemes such as the Randomized Midpoint Discretization~\citep{shen2019randomized} or the Leap-Frog Integrator~\citep{bishop2006pattern}. These algorithms can be combined with Metropolis-Hastings filters to correct the bias of the algorithm's stationary distribution; this typically improves mixing times to $\pi$ from polynomial to polylogarithmic in the inverse accuracy. A notable example is the Metropolis-Adjusted Langevin Algorithm~\citep{roberts1996exponential}, whose mixing time has been the focus of much recent work~\citep{chewi2021optimal,dwivedi2018log} and was recently characterized in~\citep{wu2021minimax}. 
While there have been many more exciting developments, this literature is far too large to attempt to comprehensively cite here, and we refer the reader to textbooks and surveys such as~\citep{andrieu2003introduction,robert1999monte,liu2001monte,roberts2004general,chewibook} for further background.
\par The present paper fits into this line of work by analyzing the (Projected Stochastic) Langevin Algorithm. We focus on this algorithm because it is arguably the simplest first-order sampling algorithm and also is the direct sampling analog of (Projected Stochastic) Gradient Descent.

\paragraph*{Langevin Algorithm analyses under different assumptions.} While the Langevin Diffusion converges asymptotically under mild tail assumptions~\citep{roberts1996exponential}, non-asymptotic convergence of the Langevin Algorithm require stronger assumptions. This is in part because sampling is in general computationally hard if the potential is non-convex, e.g., due to reductions from non-convex optimization. Non-asymptotic mixing bounds for the Langevin Algorithm were first derived by the seminal paper~\citep{Dalalyan16} in the unconstrained setting where $f$ is strongly convex and smooth. This led to a rapidly growing line of work on mixing times (primarily to $\pi$) in many settings---for example log-concave/strongly log-concave settings and  constrained/unconstrained settings as already mentioned above, but also settings where the potential 
has less stuctural assumptions, for example when it is non-smooth~\citep{durmus2019analysis,liang2021proximal,lehec2021langevin,nguyen2021unadjusted,chatterji2020langevin}, or is non-convex and only satisfies tail-growth or isoperimetric conditions~\citep{mou2022improved,raginsky2017non,erdogdu2022convergence,erdogdu2021convergence,nguyen2021unadjusted,erdogdu2018global,cheng2018sharp,balasubramanian2022towards,VempalaW19}. Another line of work has focused on sampling from non-convex manifolds~\citep{li2020riemannian,cheng2022theory,gatmiry2022convergence}.

\par We highlight a recent line of work which relaxes (strong) convexity of the potential to functional inequalities of the target distribution $\pi$. This idea was pioneered by the paper~\citep{VempalaW19}, which pointed out that smoothness and functional inequalities are sufficient to establish rapid mixing---essentially replacing convexity and strong convexity by Poincar\'e and Log-Sobolev inequality assumptions, respectively. A rapidly growing line of recent work has established even faster mixing bounds under these assumptions (e.g.,~\citep{chewi2021analysis,balasubramanian2022towards,erdogdu2022convergence,erdogdu2021convergence}) as well as found appealing ways to interpolate between these functional inequalities~\citep{chewi2021analysis}. It should be noted that for the purpose of mixing times, convexity and functional inequality assumptions are incomparable in the sense that on one hand, functional inequalities allow for (mild) non-convexity; and on the other hand, while convexity implies a Poincar\'e inequality with a finite constant, there is no uniform bound on this Poincar\'e constant. 

\par In this paper we focus on the foundational setting of smooth convex potentials $f$ (possibly also with strong convexity and/or bounded domains and/or stochastic gradients) for several reasons: (i) this is the direct sampling analog of the most classical setup in convex optimization; (ii) this setup has many applications in machine learning, statistics, and scientific computing; and (iii) optimal mixing times for $\pieta$ were previously unknown even in this seemingly simple setting.

\paragraph*{State-of-the-art mixing bounds for the Langevin Algorithm.}
As mentioned earlier, most previous work focuses on mixing to $\pi$. Existing bounds vary depending on the precise set of assumptions, the precise variant of the algorithm, and the notion of mixing error. For several common combinations thereof, state-of-the-art mixing bounds have changed in essentially each of the past few years---and may continue to change in the years to come, 
as there is no\footnote{\label{fn:fisher}
	At the time of writing, there was no setting in which current mixing bounds for the (Unadjusted) Langevin Algorithm were known to be optimal. The concurrent work~\citep{chewi2022fisher} also shows an optimality result. Their result concerns an entirely different setting (non-convex potentials where error is measured in Fisher information) and applies to a variant of the Langevin Algorithm (that outputs the averaged iterate) in the specific regime of very large mixing error (namely $\eps \approx \sqrt{d}$, where $d$ is the dimension). We also note that our lower bound (Theorem~\ref{thm:sc-lb}) establishes that an upper bound in~\citep{VempalaW19} for mixing to $\pieta$ is tight in certain metrics in the unconstrained, non-stochastic, strongly log-concave, and smooth setting. See the main text for details.
}  setting in which current mixing bounds for the (Unadjusted) Langevin Algorithm are known to be optimal.

\par For example, for the unconstrained, smooth, strongly log-concave setting originally investigated by~\citep{Dalalyan16,durmus2016sampling,durmus2017nonasymptotic,Dalalyan2017further}, a line of recent work~\citep{VempalaW19,GaneshT20,erdogdu2022convergence,chewi2021analysis,durmus2019analysis} has led to state-of-the-art mixing results to $\pi$ in R\'enyi divergence $\cD_{\alpha}$ of order roughly $\Otilde(\alpha d \kappa^2/\eps)$ for $\alpha$ of moderate size~\citep{chewi2021analysis}, and even faster $\Otilde(d \kappa / \eps)$ for the KL divergence using averaged iterates~\citep{durmus2019analysis}. Here $d$ denotes the dimension, $\eps$ denotes the mixing error to $\pi$, $m$ denotes the strong convexity of the potential $f$, $M$ denotes its smoothness, and $\kappa = M/m$ denotes the condition number. In this setting, these R\'enyi and KL guarantees recover state-of-the-art Wasserstein guarantees~\citep{Dalalyan16,Dalalyan2017further,DurmusM19} by Talagrand's inequality.

\par For this unconstrained, smooth, strongly-log concave setting, the recent revision\footnote{We thank Andre Wibisono for making us aware of this revision.} of~\citep[Lemmas 4,8]{VempalaW19} provided an upper bound of $\Otilde(\frac{\alpha}{m\eta})$ on the mixing time of the Langevin Algorithm to $\pieta$. Our lower bound (Theorem~\ref{thm:sc-lb}) establishes that their upper bound is tight for R\'enyi divergences of constant size $\alpha = \Theta(1)$. Our upper bound $\Otilde(\frac{1}{\eta m})$ for this setting (Theorem~\ref{thm:sc-ub}) has exponentially better (and optimal) dependence on the R\'enyi parameter $\alpha$, does not require as onerous a warm start, and extends to other settings (e.g., constrained, stochastic, all non-trivial stepsizes, and log-concave target distributions). 

\par For the unconstrained, smooth, (non-strongly) log-concave setting, state-of-the-art mixing results to $\pi$ in $\TV$ are of order roughly $\Otilde(d^2 M^2 c_{\textrm{PI}}^2 /\eps^4)$~\citep{balasubramanian2022towards} and $\Otilde(d^3 M^2 c_{\textrm{PI}}^2 / \eps^2) $~\citep{chewi2021analysis}. Here, $c_{\textrm{PI}}$ is the Poincar\'e constant, which is guaranteed finite but cannot be uniformly bounded. The lack of strong log-concavity poses technical challenges, leading to some results requiring averaged iterates and/or having unstable guarantees as $T$ grows.
One approach for bypassing these challenges is to consider ``modified'' algorithms which are run on a quadratically regularized (and therefore strongly convex) potential~\citep{Dalalyan16}; this has led to state-of-the-art guarantees for Wasserstein mixing~\citep{dalalyan2019bounding} and MALA~\citep{dwivedi2018log}. For further details and history, see, e.g.,~\citep[Table 1]{dwivedi2018log} and~\citep[Table 2]{chewi2021analysis}.

\par For constrained settings, much less is known. Indeed, the only mixing time result we are aware of for the Projected Langevin Algorithm is~\citep{bubeck2018sampling}, which in a tour-de-force analysis shows a polynomial mixing time to $\pi$, albeit with large polynomial factors like $d^{12}$. Faster upper bounds for mixing to $\pi$ were recently shown for different algorithms, notably the Proximal Langevin Algorithm~\citep{durmus2018efficient,brosse2017sampling} and the Mirror Langevin Algorithm~\citep{ahn2021efficient}.
The paper~\citep{bubeck2018sampling} states two open questions for the Projected Langevin Algorithm: extend their mixing results to mini-batch gradients, and obtain optimal dependence on the dimension and accuracy. We accomplish both goals for $\pieta$, and hope that this will help pave the way towards the answers for $\pi$. 

\par The present paper fits into this line of work by establishing tight mixing times to $\pieta$ for the (Projected Stochastic) Langevin Algorithm in a variety of settings: constrained/unconstrained, log-concave/strongly log-concave, and exact/stochastic gradients. An appealing aspect of our analysis technique is that it extends to all these settings with essentially no change. An appealing aspect of our results on mixing to $\pieta$ is that they clarify that the dimension and error dependencies for mixing to $\pi$ arise only through a single factor of $1/\eta$, see the discussion following Theorem~\ref{thm:intro-convex}.

\paragraph*{Lower bounds.} Whereas the study of lower bounds is well developed in the field of optimization, much less is known for sampling. We refer the reader to the recent paper~\citep{chewi2022query} (which settles this question in dimension $1$ for unconstrained, smooth, strongly log-concave sampling) for a discussion of this gap (in dimensions greater than $1$) and the technical challenges for overcoming it. Because of these challenges, existing sampling lower bounds are primarily algorithm-specific, e.g., for the Underdamped Langevin Algorithm~\citep{cao2020complexity}, or the Unadjusted Langevin Algorithm~\citep{chewi2022fisher}, or the Metropolis Adjusted Langevin Algorithm~\citep{wu2021minimax,lee2021lower,chewi2021optimal}.
\par The present paper fits into this line of work on algorithm-specific lower bounds by establishing tight lower bounds for the Langevin Algorithm to mix to $\pieta$ in the log-concave setting. These are the first tight lower bounds for the mixing time of the (Unadjusted) Langevin Algorithm (to either $\pieta$ or $\pi$). See Footnote~\ref{fn:fisher}.

\paragraph*{Connections between optimization and sampling.} 
Much of the literature on sampling is inspired by similarities to optimization.
These parallels are particularly strong between the Langevin Algorithm and Gradient Descent---the simplest and most canonical first-order algorithms for sampling and optimization, respectively. Broadly speaking, there are two approaches for extending this connection from algorithm definition to analysis. One approach views the Langevin Algorithm as Gradient Descent plus noise, and borrows classical techniques from optimization over $\R^d$ to prove sampling bounds---indeed this was the motivation of the original non-asymptotic bounds in~\citep{Dalalyan16}, see the discussion in~\citep[\S3]{Dalalyan2017further}. The other approach links sampling on $\R^d$ not with optimization on $\R^d$, but rather with optimization over the lifted space of probability measures $\cP(\R^d)$ endowed with the geometry of the $2$-Wasserstein metric. This connection is based on the celebrated JKO Theorem~\citep{jordan1998variational}, which states that the Langevin Diffusion is the gradient flow on $\cP(\R^d)$ with respect to the objective functional $\KLarg{\cdot}{\pi}$. This connection is strengthened by the fact that properties of $\pi$ such as log-concavity, strong log-concavity, and log-Sobolev inequalities imply properties of the functional $\KLarg{\cdot}{\pi}$ such as convexity, strong convexity, and gradient domination inequalities, respectively. A fruitful line of recent work has exploited this connection by analyzing the Langevin Diffusion (and its variants) via optimization techniques on Wasserstein space, see e.g.,~\citep{wibisono2018sampling,bernton2018langevin,durmus2019analysis,balasubramanian2022towards,ma2021there}.

\par The present paper fits into this interface of optimization and sampling by providing a new way to extract quantitative bounds from the first approach (linking sampling and optimization over $\R^d$). 
Specifically, on the analysis side, we show how the Privacy Amplification by Iteration technique implies tight mixing bounds for the Langevin Algorithm by exploiting the foundational fact from optimization that a Gradient Descent step is contractive (but in a different way from~\citep{Dalalyan16,Dalalyan2017further}). This leads to error bounds for sampling that in many ways mirror classical results in optimization, see Table~\ref{tab:sampling-vs-opt}.

\paragraph*{Connections to differential privacy.} The Langevin Algorithm is identical to Noisy-SGD---the standard algorithm for large-scale machine learning with differential privacy (DP)---modulo a different noise scaling and possibly gradient clipping. What it means for Noisy-SGD to satisfy DP is similar, at least intuitively, to what it means for the Langevin Algorithm to mix. Essentially, DP means that the algorithm yields indistinguishable outputs when run from identical initializations on different objectives, whereas mixing means that the algorithm yields indistinguishable outputs when run from different initializations on identical objectives. These syntactic similarities suggest that analyses might carry over if appropriately modified, and indeed this cross-pollination has proven fruitful in both directions. In one direction, ideas from sampling have recently been used to achieve tight DP guarantees for Noisy-SGD in certain regimes~\citep{ChourasiaYS21,RyffelBP22,YeS22}. In the other direction, the technique of adaptive composition from DP has recently been used to prove the first mixing bounds for the Langevin Algorithm in the R\'enyi divergence~\citep{GaneshT20,erdogdu2022convergence}. 
\par The present paper fits into this interface of sampling and DP in the sense that we introduce to the samping literature a new analysis technique from the DP literature (namely, Privacy Amplification by Iteration), and we show that this gives optimal mixing bounds for the Langevin Algorithm.

\paragraph*{Privacy Amplification by Iteration.} As mentioned earlier, the technique of Privacy Amplification by Iteration was first introduced in~\citep{pabi}. It has been used to derive fast algorithms with optimal privacy-utility tradeoffs for stochastic convex optimization~\citep{FeldmanKoTa20}. \citep{BalleBGG19} considered generalizations of this approach, and in particular showed a strengthening for strongly convex losses. \citep{Asoodeh20} prove a version of Privacy Amplification by Iteration that holds for bounded sets by using contraction coefficients for hockey-stick divergence, and~\citep{Sordello21} extend these results to shuffled SGD. Recent work by the authors~\citep{AltTal22dp} provides a diameter-aware version of Privacy Amplification by Iteration that also incorporates Privacy Amplification by Sampling for the Sampled Gaussian Noise Mechanism. This is the version of Privacy Amplification by Iteration exploited in this paper. 
\par The present paper fits into this line of work by showing how this technique, previously only used in the privacy literature, also enables understanding questions in the sampling literature.

\subsection{Organization}\label{ssec:intro-outline}

\S\ref{sec:prelim} recalls relevant preliminaries.
\S\ref{sec:convex} and \S\ref{sec:sc} contain our main results, which characterize the mixing time of the Langevin Algorithm for log-concave and strongly log-concave targets, respectively.
\S\ref{sec:discussion} describes several directions for future research that are motivated by the results in this paper. For brevity, several proofs and auxiliary results are deferred to the Appendix.

\section{Preliminaries}\label{sec:prelim}

In this section, we recall relevant preliminaries about regularity properties in \S\ref{ssec:prelim:convexity}, 
the R\'enyi divergence in \S\ref{ssec:prelim:renyi}, and the shifted R\'enyi divergence and the analysis technique of Privacy Amplification by Iteration in \S\ref{ssec:prelim:pabi}. In an effort towards self-containedness as well as accessibility of this article within the sampling community, the Appendix contains proofs of the convex optimization lemma in \S\ref{ssec:prelim:convexity} and of the differential privacy lemmas in \S\ref{ssec:prelim:pabi}.

\paragraph*{Notation.} Throughout $\|\cdot\|$ denotes the Euclidean norm on $\R^d$, and $\cP(\R^d)$ denotes the set of Borel-measurable probability distributions over $\R^d$. A function $f : \R^d \to \R$ is said to be $c$-contractive (equivalently, $c$-Lipschitz) if $\|f(x) - f(y)\| \leq c \|x - y\|$ for all $x,y \in \R^d$. If $c =1$, then we simply say $f$ is contractive; and if $c < 1$, then we say $f$ is $c$-strongly contractive. Throughout, the potentials are assumed differentiable so that the Langevin Algorithm is well-defined, and we assume $\int e^{-f(x)} dx < \infty$ so that the target distribution $\pi$ is well-defined. Also, throughout $\cK$ is a non-empty, closed, convex subset of $\R^d$ and for brevity we drop the quantifiers ``non-empty'' and ``closed''. These assumptions ensure that the projection operator $\proj(x) := \argmin_{y \in \cK} \|y - x\|$ onto $\cK$ is well-defined (see, e.g.,~\citep[Theorem 3.1.10]{nesterov2003introductory}). All other notation is introduced in the main text.

\subsection{Convexity, Smoothness, and Contractivity}\label{ssec:prelim:convexity}

The convergence of the Langevin Algorithm depends, of course, on the regularity properties of the potentials. We recall two of the most basic assumptions below: strong convexity and smoothness, which amount to under and overestimates, respectively, of the second-order Taylor expansion. Or, if the function is twice differentiable, then $m$-strong convexity and $M$-smoothness are equivalent to the simple eigenvalue condition $mI_d \preceq \nabla^2 f(x) \preceq MI_d$ for all $x$. Note also that $0$-strong convexity corresponds to convexity.

\begin{defin}[Regularity assumptions on potentials]\label{def:regularity}
	A differentiable function $f : \R^d \to \R$ is:
	\begin{itemize}
		\item $m$-strongly convex if $f(x) \geq f(y) + \langle \nabla f(y), x-y\rangle + \frac{m}{2}\|x-y\|^2$ for all $x,y \in \R^d$.
		\item $M$-smooth if $f(x) \leq f(y) + \langle \nabla f(y), x - y\rangle + \frac{M}{2}\|x-y\|^2$ for all $x,y \in \R^d$.
	\end{itemize}
	For settings in which potentials are both $m$-strongly convex and $M$-smooth, the corresponding condition number is denoted by $\kappa := M/m \geq 1$. 
\end{defin}

Since these concepts are by now standard in the sampling literature, for brevity we refer to the monograph~\citep{nesterov2003introductory} for a detailed discussion of them and their many equivalent formulations (e.g., a differentiable function $f$ is $M$-smooth if and only if $\nabla f$ is $M$-Lipschitz). However, we do recall here one fundamental consequence of these properties that we use repeatedly in the sequel:
smoothness implies contractivity of Gradient Descent updates. This improves to strong contractivity under the additional assumption of strong convexity. This contractivity property is arguably one of the most important consequence of strong convexity and smoothness in convex optimization. Since this fact is key to our results, for the convenience of the reader we provide a short proof in Appendix~\ref{app:contractivity}.

\begin{lemma}[Contractivity of Gradient Descent updates]\label{lem:contractive-gd}
	Suppose $f$ is an $m$-strongly convex, $M$-smooth function for $0 \leq m \leq M < \infty$. For any stepsize $\eta \leq 2/M$, the Gradient Descent update 
	\[x \mapsto x - \eta \nabla f(x)\]
	is $c$-contractive for 
	\[
		c = \max_{\lambda \in \{m,M\}} \abs{1 - \eta \lambda}.
	\]
	In particular, if $ 0 < m \leq M < \infty$, then $c = (\kappa-1)/(\kappa+1)$ for the stepsize choice $\eta = 2/(M+m)$.
\end{lemma}

In the sequel, we also make use of the basic fact from convex analysis that the projection operator onto a convex set is contractive. A proof can be found in, e.g.,~\citep[Theorem 1.2.1]{schneider2014convex}. 

\begin{lemma}[Contractivity of projections onto convex sets]\label{lem:contractive-proj} If $\cK \subset \R^d$ is convex, then $\proj$ is contractive.
\end{lemma}

\subsection{R\'enyi Divergence}\label{ssec:prelim:renyi}

Here we define the R\'enyi divergence as it plays a central role in our analysis. See the survey~\citep{van2014renyi} for a detailed discussion of R\'enyi divergences and their many properties. In what follows, we adopt the standard convention that $0/0 = 0$ and $x/0 = \infty$ for $x \neq 0$.
	
\begin{defin}[R\'enyi divergence]\label{def:rd}
	The R\'enyi divergence between probability distributions $\mu,\nu \in \cP(\R^d)$ of order $\alpha \in (1,\infty)$ is
	\[
		\Dalp{\mu}{\nu} = \frac{1}{\alpha - 1} \log \int \left( \frac{\mu(x)}{\nu(x)} \right)^{\alpha} \nu(x) dx
	\]
	if $\mu \ll \nu$, and is $\infty$ otherwise. The R\'enyi divergences of order $\alpha \in \{1,\infty\}$ are defined by continuity. 
\end{defin}

Our analysis makes use of the following facts about the R\'enyi divergence. First, we recall standard comparison inequalities that upper bound various probability metrics by the R\'enyi divergence. Since we prove fast mixing with respect to R\'enyi divergence, these inequalities immediately imply fast mixing with respect to all these other metrics. Note that for the KL and Chi-Squared divergences, no inequality is needed, since these are equal to the R\'enyi divergence with parameters $\alpha = 1$ and $2$, respectively (modulo a simple transform for the latter). The Total Variation bound is Pinkser's inequality, and the Hellinger bound can be found in, e.g.,~\citep[equation (7)]{van2014renyi}.

\begin{lemma}[Standard comparison inequalities to R\'enyi divergence]\label{lem:renyi-ineq}
	Suppose $\mu \ll \nu$. Then:
	\begin{itemize}
		\item \underline{Kullback-Leibler divergence:} $\KLarg{\mu}{\nu} = \cD_{1}(\mu \; \| \; \nu)$.
		\item \underline{Chi-squared divergence:} $\chi^2(\mu, \nu) = \exp\left(\cD_{2}(\mu \; \| \; \nu)\right) - 1$.
		\item \underline{Total variation distance:} $\TV(\mu,\nu) \leq \sqrt{\half \, \cD_{1}(\mu \; \| \; \nu)}$.
		\item \underline{Hellinger divergence:} $\Hell(\mu,\nu) \leq \sqrt{\cD_{1}(\mu \; \| \; \nu)}$.
	\end{itemize}
\end{lemma}

We also make use of the data-processing inequality for R\'enyi divergence. This generalizes the standard data-processing inequality for the KL divergence to all $\alpha \geq 1$. A proof can be found in, e.g.~\citep[Theorem 9]{van2014renyi}.

\begin{lemma}[Data-processing inequality for R\'enyi divergence]\label{lem:renyi-dataprocess}
	For any R\'enyi divergence parameter $\alpha \geq 1$, any (possibly random) function $h : \R^d \to \R^{d'}$, and any distributions $\mu,\nu \in \cP(\R^d)$,
	\[
		\Dalplr{h_{\#} \mu}{h_{\#}\nu} \leq \Dalplr{\mu}{\nu}\,.
	\]
\end{lemma}

The next fact is a closed-form expression for the R\'enyi divergence between two Gaussian distributions. This identity can be found, e.g., in~\citep[equation (10)]{van2014renyi}.

\begin{lemma}[R\'enyi divergence between Gaussians]\label{lem:renyi-gaussians}
	If $\sig_{\alpha}^2 = (1 - \alpha) \sig_0^2 + \alpha \sigma_1^2$ is non-zero, then
	\[
		\Dalplr{ \cN\left( \mu_0, \sig_0^2 \right)}{ \cN\left( \mu_1, \sig_1^2 \right)}
		=
		\frac{\alpha(\mu_1 - \mu_0)^2}{2\sig_\alpha^2} + \frac{1}{1-\alpha} \log \frac{\sigma_\alpha}{ \sigma_0^{1-\alpha} \sigma_1^{\alpha} }.
	\]
\end{lemma}

\subsection{Shifted R\'enyi Divergence and Privacy Amplification by Iteration}\label{ssec:prelim:pabi}

Here we describe the core technique in our mixing time analysis: Privacy Amplification by Iteration. This technique was originally introduced in~\citep{pabi}; however, for the purposes of this paper, it is crucial to use the diameter-aware variant of this technique which was recently developed in~\citep{AltTal22dp}. We note that everything in this subsection can be extended to arbitrary noise distributions over Banach spaces, but for simplicity of exposition we restrict to Gaussian noise over $\R^d$ as this suffices for the present purpose of mixing time bounds.

\par We begin by defining the shifted R\'enyi divergence, which is the key quantity that Privacy Amplification by Iteration concerns. In words, the shifted R\'enyi divergence is a version of the R\'enyi divergence that is smoothed out by the Optimal Transport (a.k.a., Wasserstein distance). This smoothing makes the R\'enyi divergence geometrically aware in a natural way for mixing analyses, see the discussion in the techniques section \S\ref{ssec:intro-tech}.

\begin{defin}[Shifted R\'enyi divergence]\label{def:srd}
	The shifted R\'enyi divergence between probability distributions $\mu,\nu \in \cP(\R^d)$ of order $\alpha \in [1,\infty]$ and shift $z \geq 0$ is
	\[
	\Dalpshiftlr{z}{\mu}{\nu} = \inf_{\mu' \; : \; W_{\infty}(\mu,\mu') \leq z} \Dalplr{\mu'}{\nu}.
	\]
\end{defin}

The technique of Privacy Amplification by Iteration uses the shifted R\'enyi divergence as an intermediate Lyapunov function in order to prove that the (unshifted) R\'enyi divergence between the final two iterates of two processes is close, so long as these processes are generated by interleaving contraction maps and additive Gaussian noise. This notion is formalized as follows.

\begin{defin}[Contractive Noisy Iteration]\label{def:cni}
	Consider an initial distribution $\mu_0 \in \cP(\R^d)$, a sequence of (random) $c$-contractions $\phi_t : \R^d \to \R^d$, and a sequence of noise distributions $\xi_t$. The $c$-Contractive Noisy Iteration ($c$-CNI) is the process generated by
	\[
	X_{t+1} = \phi_{t+1}(X_t) + Z_t\,,
	\]
	where $X_0 \sim \mu_0$ and $Z_t \sim \xi_t$ independently. For shorthand, we denote the law of the final iterate $X_T$ by $\CNI_c(\mu_0,\{\phi_t\},\{\xi_t\})$.
\end{defin}

The Privacy Amplification by Iteration technique (or, more precisely, the modification thereof in~\citep{AltTal22dp}) provides the following key estimate. 
We remark that while we state here separate bounds for the cases $c < 1$ and $c=1$, the true PABI bound is slightly sharper and is continuous in $c$, see Remark~\ref{rem:pabi-continuous}---however, we ignore this sharpening because it complicates the expressions and is uneceessary for the purposes of this paper.

\begin{prop}[Diameter-aware Privacy Amplication by Iteration bound from~\citep{AltTal22dp}, simplified version]\label{prop:pabi-new}
	Suppose $X_T \sim \CNI_c(\mu_0,\{\phi_t\},\{\xi_t\})$ and $X_T' \sim \CNI_c(\mu_0',\{\phi_t\},\{\xi_t\})$ where the initial distributions satisfy $W_{\infty}(\mu_0,\mu_0') \leq D$, the update functions $\phi_t$ are $c$-contractive, and the noise distributions $\xi_t = \cN(0,\sig^2 I_d)$ are Gaussian. Then
	\begin{align*}
		\Dalplr{X_T}{X_T'} 
		\leq \frac{\alpha D^2}{2\sig^2} \cdot 
		\begin{cases}
			1/T & \text{ for }c = 1
			\\ c^{2T}   & \text{ for }c < 1
		\end{cases}
	\end{align*}
\end{prop}

The proof of this proposition is a simple combination of two lemmas, which disentangle how the shifted R\'enyi divergence is affected by the contraction and noise operations defining the CNI. See the techniques section \S\ref{ssec:intro-tech} for a high-level overview. Because these arguments have previously only appeared in the differential privacy literature and also because they can be simplified\footnote{
	The Privacy Amplification by Iteration bound~\citep[Proposition 3.2]{AltTal22dp} from DP generalizes Proposition~\ref{prop:pabi-new} in two ways: the update functions $\phi_t$ in the two CNI are different, and the bound uses an intermediate time between $0$ and $T$. The simplifications are possible here because for mixing, the goal is to understand the stability of the same algorithm when run on two different initializations; see \S\ref{ssec:intro-prev} for a discussion of the differences between DP and mixing.
 } for the present purpose of mixing time analyses, in the interest of accessibility to the sampling community, we recall in Appendix~\ref{app:pabi} these two lemmas, their proofs, and how they imply Proposition~\ref{prop:pabi-new}.

\section{Mixing Time for Log-Concave Distributions}\label{sec:convex}

In this section we establish tight mixing bounds on the Langevin Algorithm to its stationary distribution $\pieta$ under the assumption that the target distribution is log-concave (equivalently, the potentials are convex).

\begin{theorem}[Upper bound: mixing of the Langevin Algorithm for log-concave targets]\label{thm:convex-ub}
	Consider any convex set $\cK \subset \R^d$ of diameter $D$, any $M$-smooth convex potentials $f_i : \cK \to \R$, any stepsize $\eta \leq 2/M$, and any batch size $b \in [n]$. The Langevin Algorithm mixes to its stationary distribution $\pieta$ in time
	\[
		T_{\mathrm{mix},\, \TV}\left( \frac{1}{4} \right) \leq \frac{2D^2}{\eta}\,.
	\]
\end{theorem}

The next result proves that this mixing bound is tight up to a constant factor. No attempt has been made to optimize this constant.

\begin{theorem}[Lower bound: mixing of the Langevin Algorithm for log-concave targets]\label{thm:convex-lb}
	Consider running the Langevin Algorithm with any stepsize $\eta > 0$ and any batch size $b \in [n]$ on the identically zero potentials $f_i \equiv 0$ over the interval $\cK  = [-D/2,D/2] \subset \R$. 
	Then
	\[
	T_{\mathrm{mix},\,\TV}\left( \frac{1}{4} \right)
	\geq \frac{D^2}{100\eta}\,.
	\]
\end{theorem}

We remark that the fast mixing in total variation in Theorem~\ref{thm:convex-ub} implies the following fast mixing in other notions of distance via standard comparison inequalities and the fact that the mixing contraction rate for total variation is worst-case among all $f$-divergences; details in \S\ref{app:otherdists}. 

\begin{cor}[Mixing bounds for other notions of distance]\label{cor:convex-mixing}
	Consider the setup in Theorem~\ref{thm:convex-ub}. 
	The Langevin Algorithm mixes to $\eps$ error in:
	\begin{itemize}
		\item $T_{\mathrm{mix},\, D_{\alpha}}(\eps) \lesssim \frac{D^2}{\eta}(\alpha + \logeps)$.
		\item $T_{\mathrm{mix},\, \KL}(\eps) \lesssim \frac{D^2}{\eta} \logeps$.
		\item $T_{\mathrm{mix},\, \chi^2}(\eps) \lesssim \frac{D^2}{\eta} \logeps$.
		\item $T_{\mathrm{mix},\, \TV}(\eps) \lesssim \frac{D^2}{\eta} \logeps$.
		\item $T_{\mathrm{mix},\, \Hell}(\eps) \lesssim \frac{D^2}{\eta } \logeps$.
	\end{itemize}
\end{cor}

\subsection{Upper Bound (Proof of Theorem~\ref{thm:convex-ub})}\label{ssec:convex-ub}

		\paragraph*{Step 1: Coupling the iterates.} Consider running the Langevin Algorithm for $T$ iterations from two different initializations $X_0 \sim \mu_0$ and $X_0' \sim \pieta$, where $\mu_0$ is an arbitrary probability distribution supported on $\cK$, and $\pieta$ is the stationary distribution of the algorithm. Then by definition of the algorithm's updates, we can couple the random batches $\{B_t\}_{t=0}^{T-1}$ and the noise $\{Z_t\}_{t=0}^{T-1}$ so that
	\begin{align*}
		X_{t+1} &= \projlr{X_t  - \frac{\eta}{b} \sum_{i \in B_t} \nabla f_i(X_t)  + Z_t} 
		\\  
		X_{t+1}' &= \projlr{X_t' - \frac{\eta}{b} \sum_{i \in B_t}  \nabla f_i( X_t') + Z_t}
	\end{align*}
	where $Z_t \sim \cN(0,2\eta I_d)$ are drawn independently. 
	
	\paragraph*{Step 2: Construction of auxiliary Contractive Noisy Iterations.}\footnote{This construction allows us to combine the two contraction maps (projection and stochastic gradient update) into a single contraction, so that the new sequences are bona fide CNI. Alternatively, one can directly argue about the original $\{X_t\}$, $\{X_t'\}$ sequences as ``CNCI'' or ``Projected CNI'' using the contraction-reduction lemma twice to analyze each iteration, as described in~\citep{AltTal22dp}. We adopt this shifting here for simplicity of exposition.} Define auxiliary sequences $\{Y_t\}_{t=0}^T$ and $\{Y'_t\}_{t=0}^T$ as follows. Initialize $Y_0 = X_0$ and $Y'_0 = X'_0$. Then update via
	\begin{align*}
		Y_{t+1} &= \phi_{t+1}(Y_t) + Z_{t}\\
		Y'_{t+1} &= \phi_{t+1}(Y'_t) + Z_{t}
	\end{align*}
	using the same update function, which is defined as
	\begin{align*}
		\phi_{t+1}(y) &= \projlr{y} - \frac{\eta}{b} \sum_{i \in B_t} \nabla f_i(\projlr{y}).
	\end{align*}
	Observe that by induction, these auxiliary processes satisfy $X_t = \projlr{Y_t}$ and $X_t' = \projlr{Y_t'}$.
	
	In order to show that these two auxiliary processes are $1$-CNI (henceforth simply called CNI), it suffices to check that the stochastic update functions $\phi_t$ are contractive, as shown next. 
	
	\begin{obs}[Contractivity of $\phi_t$]\label{obs:convex-contractive-update}
		Suppose the potentials $f_i$ are convex and $M$-smooth. If $\eta \leq 2/M$, then $\phi_t$ is contractive.
	\end{obs}
	\begin{proof}
			The claim follows because $\phi_t$ is the composition of a projection and a stochastic gradient descent update, both of which are contractive. Formally, consider any $y,y' \in \R^d$; we show that $\|\phi_t(y) - \phi_t(y')\| \leq \|y - y'\|$. For shorthand, denote $x = \projlr{y}$ and $x' = \projlr{y'}$. Then
		\begin{align}
			\big\| \phi_t(y) - \phi_t(y') \big\|
			&=
			\big\| \big( x - \frac{\eta}{b} \sum_{i \in B_t} \nabla f_i(x) \big) - \big( x' - \frac{\eta}{b} \sum_{i \in B_t} \nabla f_i(x') \big) \big\| \nonumber
			\\ &\leq
			\frac{1}{b} \sum_{i \in B_t} \big\| \left( x - \eta \nabla f_i(x) \right) - \left( x' - \eta \nabla f_i(x') \right) \big\| \nonumber
			\\ &\leq
			\|x - x'\| 		\label{eq:convex-contractive-update}
			\\ &= 
			\|\projlr{y} - \projlr{y'}\| \nonumber
			\\ &\leq \|y - y'\|\,. \nonumber
		\end{align}
			Above, the first inequality uses the triangle inequality, the second inequality uses the fact that a Gradient Descent update is contractive (Lemma~\ref{lem:contractive-gd}), and the third inequality uses the fact that the projection function onto a convex set is contractive (Lemma~\ref{lem:contractive-proj}).
	\end{proof}

	\paragraph*{Step 3: Fast mixing via Privacy Amplification by Iteration.} Since we have established that the two auxiliary sequences $\{Y_t\}_{t=0}^{T}$ and $\{Y_t'\}_{t=0}^T$ are CNI with identical update functions, Proposition~\ref{prop:pabi-new} with $\alpha = 1$ implies
	\[
		\KLlrarg{Y_{T}}{Y_{T}'} \leq \frac{D^2}{4\eta T}\,.
	\]
Now make two observations. First, by the data-processing property of the KL divergence (Lemma~\ref{lem:renyi-dataprocess}) and the fact that $X_T = \projlr{Y_T}$ and $X'_T = \projlr{Y'_T}$, we have $\KLlrarg{X_T}{X_T'} \leq \KLlrarg{Y_{T}}{Y_{T}'}$. Second, by stationarity of $X_0' \sim \pieta$, we have by induction that $X_T' \sim \pieta$. It follows that
	\begin{align}
		\KLlrarg{X_T}{\pieta}
		=
		\KLlrarg{X_T}{X_T'}
		\leq 
		\KLlrarg{Y_{T}}{Y_{T}'} 
		\leq
		\frac{D^2}{4\eta T}\,.
		\label{eq:convex-proof-constant}
	\end{align}
	Therefore the KL divergence is at most $1/8$ after $T = 2D^2/\eta$ iterations. Since a KL divergence bound of $1/8$ implies a TV bound of $1/4$ by Pinkser's inequality (Lemma~\ref{lem:renyi-ineq}), we conclude that
	\begin{align}
		T_{\mathrm{mix},\,\TV}(1/4) \leq \frac{2D^2}{\eta}\,.
		\label{eq:convex-proof-TV-partial}
	\end{align}

\subsection{Lower Bound (Proof of Theorem~\ref{thm:convex-lb})}\label{ssec:convex-lb}

For these identically zero potentials, the update of the Langevin Algorithm simplifies to
\[
X_{t+1} = \Pi_{[-D/2,D/2]} \left[X_t + Z_t \right]
\]
where $Z_t \sim \cN(0,2\eta)$. Observe that the stationary distribution $\pieta$ puts half the mass on the set $[0,D/2]$ by symmetry. (Indeed, if the the Markov chain were initialized at $0$, then by induction the law of each iterate is symmetric around $0$, hence the stationary distribution is too.)

\par Thus, in order for the Langevin Algorithm to mix to total variation error $\TV(X_T, \, \pieta) \leq 1/4$, it must be the case that
\begin{align}
	\Prob\big[ X_T \geq 0 \big] \geq 1/4.
	\label{eq:thm-lb-convex-bound}
\end{align}
Therefore, in order to prove the present theorem, it suffices to show that this inequality~\eqref{eq:thm-lb-convex-bound} is false for initialization $X_0 = -D/4$. To this end, observe that
\[
\Prob\Bigg[ X_T \geq 0 \Bigg]
\leq 
\Prob\Bigg[ \max_{t \leq T} \abs{\sum_{s=1}^t Z_s} \geq \frac{D}{4} \Bigg]
\leq
\exp\Bigg( - \frac{D^2}{64 T\eta } \Bigg).
\]
Above, the first step is because under the assumption that $X_0 = -D/4$ and $\max_{t \leq T} \abs{\sum_{s=1}^t Z_s} < D/4$, there is no projection and thus $X_t = -D/4 + \sum_{s=1}^t Z_s$ stays within the interval $(-D/2, 0)$. The second step is by a standard martinginale concentration inequality on the supremum of a discrete Gaussian random walk (Lemma~\ref{lem:rw-helper} below). We conclude in particular that if $T < aD^2/\eta$ for the choice of constant $a = 1/(64 \ln 4) \geq 1/100$, then $\Prob[X_t \in S] < 1/4$, which contradicts~\eqref{eq:thm-lb-convex-bound} and thus completes the proof of Theorem~\ref{thm:convex-lb}.

It remains only to state the concentration inequality used above. This fact can be proved, e.g., as a consequence of Doob's Submartingale Inequality, see~\citep[Page 139]{williams1991probability}.

\begin{lemma}[Concentration inequality for supremum of random walk]\label{lem:rw-helper}
	Suppose $\xi_1, \dots, \xi_T \sim \cN(0,1)$ are independent. Then for any $ a > 0$,
	\[
	\Prob\left( \max_{t \leq T} \abs{ \sum_{s=1}^t \xi_s} \geq a \right)
	\leq \exp\left( - \frac{a^2}{2T} \right).
	\]
\end{lemma}

\section{Mixing Time for Strongly Log-Concave Distributions}\label{sec:sc}

In this section we establish tight mixing bounds on the Langevin Algorithm to its stationary distribution $\pieta$ under the assumption that the target distribution is \emph{strongly} log-concave (equivalently, the potentials are \emph{strongly} convex). In this case, the rapid mixing improves from polynomially to exponentially fast in the diameter $D$. Following is the analog of Theorem~\ref{thm:convex-ub} and Corollary~\ref{cor:convex-mixing}.

\begin{theorem}[Upper bound: mixing of the Langevin Algorithm for strongly log-concave targets]\label{thm:sc-ub}
	Consider any convex set $\cK \subset \R^d$ of diameter $D$, any $m$-strongly convex and $M$-smooth potentials $f_i : \cK \to \R$, any stepsize $\eta < 2/M$, and any batch size $b \in [n]$. Then the Langevin Algorithm mixes to its stationary distribution $\pieta$ in time
	\begin{itemize}
		\item $T_{\mathrm{mix},\, D_{\alpha}}(\eps) \lesssim \log_{1/c} \frac{\alpha D}{\eta \eps}$.
		\item $T_{\mathrm{mix},\, \KL}(\eps) \lesssim \log_{1/c} \frac{D}{\eta \eps}$.
		\item $T_{\mathrm{mix},\, \chi^2}(\eps) \lesssim \log_{1/c} \frac{D}{\eta \eps}$.
		\item $T_{\mathrm{mix},\, \TV}(\eps) \lesssim \log_{1/c} \frac{D}{\eta \eps}$.
		\item $T_{\mathrm{mix},\, \Hell}(\eps) \lesssim \log_{1/c} \frac{D}{\eta \eps}$.
	\end{itemize}
	where $c = \max_{\lambda \in \{m,M\}} |1 - \eta \lambda|$.
\end{theorem}

\begin{remark}[Special cases of contraction coefficient $c$]
It always holds that $c < 1$. In the setting $\eta < 1/M$, the contraction coefficient simplifies to $c = 1 - \eta m \leq \exp(-\eta m)$, and thus the mixing time is of order $\tfrac{1}{\eta m} \log \tfrac{D}{\eta \eps}$. Another important case is the stepsize choice $\eta = 2/(M+m)$ which minimizes the contraction coefficient $c$, in which case $c = (\kappa-1)/(\kappa+ 1) \leq 1 - 1/\kappa \leq \exp(-\kappa)$ and thus the mixing time is of the order $\kappa \log \tfrac{D}{\eta \eps}$, where $\kappa := \tfrac{M}{m}$ denotes the condition number.
\end{remark}

We also show a lower bound on the mixing which decays at the same exponential rate $c^{\Theta(T)}$ as the upper bound in Theorem~\ref{thm:sc-ub}. Thus our upper and lower bounds on the mixing time match up to a logarithmic factor. For simplicity, we make two choices when stating this lower bound. First, we state it for the unconstrained setting since the upper bound in Theorem~\ref{thm:sc-ub} extends to unconstrained settings at the expense of at most a logarithmic factor, see Appendix~\ref{app:diam}. Second, we state this lower bound for mixing in R\'enyi divergence; this implies analogous lower bounds in divergences such as KL and chi-squared, since those are equal to certain R\'enyi divergences (modulo simple transformations). Analogous mixing time lower bounds in other notions of distance (e.g., $\TV$ or Hellinger) are easily proved by using the same lower bound construction and changing only the final step in the proof which bounds $d(\cN(0,1-c^{2T}),\, \cN(0,1))$ in the relevant metric $d$.

\begin{theorem}[Lower bound: mixing of the Langevin Algorithm for strongly log-concave targets]\label{thm:sc-lb}
	For any stepsize $\eta > 0$, there exist univariate quadratic potentials $f_i$ over $\cK = \R$ that are $m$-strongly convex, $M$-smooth, and satisfy the following. For any R\'enyi divergence parameter $\alpha \geq 1$, any number of iterations $T \in \N$, and any batch size $b \in [n]$, the $T$-th iterate $X_t$ of the Langevin Algorithm mixes to its stationary distribution $\pieta$ no faster than
	\[
	\Dalp{X_t}{ \pieta } \geq \frac{\alpha c^{4T}}{4} \,,
	\]
	where $c = \max_{\lambda \in \{m,M\}} |1 - \eta \lambda|$. Thus in particular, 
	\[
				T_{\mathrm{mix},\, D_{\alpha}}(\eps) \gtrsim \log_{1/c} \frac{\alpha}{\eps}\,.
	\]
\end{theorem}

\subsection{Upper Bound (Proof of Theorem~\ref{thm:sc-ub})}\label{ssec:sc-ub}

	The analysis is essentially identical to the analysis of log-concave distributions in Theorem~\ref{thm:convex-ub}. The only difference is that in the present setting of strongly convex potentials, the stochastic update functions $\phi_t$ are not just contractive, but in fact $c$-strongly contractive (proven in Observation~\ref{obs:sc-contractive-update} below). In other words, the auxiliary sequences $\{Y_t\}_t$ and $\{Y_t'\}_t$ are generated by $c$-CNI for $c < 1$.
	Thus the Privacy Amplification by Iteration Bound for $c$-CNI (Proposition~\ref{prop:pabi-new}) implies
	\[
		\Dalplr{Y_{T}}{Y_{T}'} \leq \frac{\alpha D^2 c^{2T}}{4\eta} \,.
	\]
The rest of the proof proceeds identically to give the desired mixing time for the R\'enyi divergence. This immediately implies the fast mixing bounds in all the other desired metrics by the standard comparison inequalities in Lemma~\ref{lem:renyi-ineq}. 
\par (Note that in the present strongly log-concave setting, there is no need to argue indirectly through the Total Variation metric because here the R\'enyi mixing bound obtained by the Privacy Amplification by Iteration argument is already exponentially fast, whereas it is not in the log-concave setting, c.f., Appendix~\ref{app:otherdists}.)
	\par It remains only to prove the strong contractivity of $\phi_t$ in the setting of strongly convex $f_i$.

\begin{obs}[Strong contractivity of $\phi_t$]\label{obs:sc-contractive-update} 
	Assume the functions $f_i$ are $m$-strongly convex and $M$-smooth. If $\eta < 2/M$, then $\phi_t$ is $c$-strongly contractive for $c = \max_{\lambda \in \{m, M\}} |1-\eta\lambda|$.
\end{obs}
\begin{proof}
	Identical to the proof of Observation~\ref{obs:convex-contractive-update}, except that the inequality~\eqref{eq:convex-contractive-update} tightens by a factor of $c$ because by Lemma~\ref{lem:contractive-gd}, a Gradient Descent update is $c$-strongly contractive rather than just contractive in the present setting where $f$ is strongly convex rather than just convex.
\end{proof}

\subsection{Lower Bound (Proof of Theorem~\ref{thm:sc-lb})}\label{ssec:sc-lb}

	Consider the univariate quadratic potential
	\[
		f(x) = \frac{\lambda}{2}x^2\,,
	\]
	for the choice of $\lambda = \argmax_{\lambda \in \{m,M\}} |1 - \eta \lambda|$. Clearly this potential is $m$-strongly convex and $M$-smooth as its second derivative is $f''(x) = \lambda \in [m,M]$. For this potential $f$ and the present unconstrained setting $\cK = \R$, the update step~\eqref{eq:la} of the Langevin Algorithm simplifies to:
	\[
		X_{t+1} = (1 - \eta \lambda) X_t + Z_t\,,
	\]
	where $Z_t \sim \cN(0,2\eta)$. By unrolling this update and considering the initialization $X_0 = 0$ (such a choice suffices for the present purpose of proving a lower bound on mixing), we have that
	\[
		X_{t+1} = \sum_{s=0}^t (1-\eta \lambda)^{t-s} Z_s\,.
	\] 
	Now by our choice of $\lambda$ as well as the basic fact that the sum of independent Gaussians is Gaussian with variance equal to the sum of the constitutents' squared variances, it follows that
	\begin{align}
		X_T
		\sim
		\cN\left( 0, 2\eta \sum_{s=0}^{T-1} (1-\eta \lambda)^{2(T-1-s)} \right)
		=
		\cN\left( 0, 2\eta \, \frac{1-c^{2T}}{1-c^2} \right)
		\label{eq:sc-lb:XT-law}
	\end{align}
	Since the Langevin Algorithm has stationary distribution $\pi_{\eta}$, taking the limit $T \to \infty$ implies
	\begin{align}
		\pieta
		=
		\cN\left( 0, 2\eta \, \frac{1}{1-c^2} \right)
		\label{eq:sc-lb:Xinf-law}
	\end{align}
	Now that we have explicitly computed the laws of intermediate iterates $X_T$  and the stationary distribution $\pieta$, both as Gaussians, it is straightforward to compute their R\'enyi divergence:
	\begin{align}
		\Dalplr{X_T}{\pieta}
		&= 
		\Dalplr{\cN\left( 0, 2\eta \, \frac{1-c^{2T}}{1-c^2} \right)}{\cN\left( 0, 2\eta \, \frac{1}{1-c^2} \right)} \nonumber
		\\ &=
		\Dalplr{\cN\left( 0, 1-c^{2T} \right)}{\cN\left( 0, 1\right)} \nonumber
		\\ &=
		\frac{1}{2\beta} \log \left(1 - \beta c^{2T} \right) - \frac{1}{2} \log \left( 1 - c^{2T} \right)\,.\label{eq:sc-lb:renyi-long}
	\end{align}
	Above, the first step is by plugging in the laws of $X_t$ and $\pieta$ from~\eqref{eq:sc-lb:XT-law} and~\eqref{eq:sc-lb:Xinf-law}; the second step is by a change of measure; the third step is by the explicit formula for the R\'enyi divergence between Gaussians (Lemma~\ref{lem:renyi-gaussians}) and simplifying. In this final step, we denote $\beta := 1 - \alpha \leq 0$ for shorthand.
	\par The desired asymptotics $\Dalplr{X_t}{\pieta} \approx \alpha c^{4T} / 2$ now follow from~\eqref{eq:sc-lb:renyi-long} by using the second-order Taylor expansion $\log(1 + \eps ) \approx \eps - \eps^2/2$ to expand both logarithmic terms. Formally, use the elementary inequality $\log (1 + x) \geq x - x^2/2$ for $x \geq 0$ to bound the first term, and use the elementary inequality $\log(1 - x) \leq - x - x^2/2$ for $x \geq 0$ to bound the latter term. We conclude that
	\[
		\Dalplr{X_t}{\pieta}
		\geq
		\frac{1}{2\beta} \left( - c^{2T} - \frac{\beta c^{4T}}{2}  \right) 
		-
		\frac{1}{2} \left( 
		- c^{2T} - \frac{c^{4T}}{2}
		\right)
		=
		\frac{(1 - \beta) c^{4T}}{4}
		=
		\frac{\alpha c^{4T}}{4}\,.
	\]

\section{Discussion}\label{sec:discussion}

We mention a few natural directions for future work that are motivated by the results of this paper.
\begin{itemize}
    \item \underline{Bias of the Langevin Algorithm?} The purpose of this paper is to characterize the mixing time of the Langevin Algorithm to its stationary distribution $\pieta$, for any discretization stepsize $\eta > 0$. In order to sample from the target distribution $\pi$, these mixing bounds must be combined with a discretization bound on the error between $\pieta$ and $\pi$. While bias bounds are known for certain settings (e.g.,~\citep{bubeck2018sampling,VempalaW19,balasubramanian2022towards,durmus2019analysis,chewi2021analysis}), for the vast majority of settings it remains unclear if the current bounds are tight, and at least in certain settings it seems clear that they are not. Our result on the mixing time of the Langevin Algorithm further motivates the quest for optimal bias bounds. 
    \par A related challenge is whether bias bounds can obtained by ``directly'' comparing the stationary distributions $\pi$ and $\pieta$---in contrast to all known bias bounds, which are obtained by taking non-asymptotic mixing bounds to $\pi$ (which compare $\pi$ to the law $\pi_{\eta,T}$ of the $T$-th iterate of the Langevin Algorithm) and sending $T \to \infty$. It seems that any such bias bound is useless for combining with our mixing bounds to $\pieta$, since the indirectness of this argument makes the resulting mixing bound (comparing $\pi_{\eta,T}$ to $\pi$) worse than the original mixing bound used to bound the bias. In fact, directly arguing about properties of the stationary distribution $\pieta$ appears to require new techniques (see the related open questions in~\citep{VempalaW19}). The only results in this vein that we are aware of are in the recent revision to~\citep{VempalaW19} and in the forthcoming work~\citep{AltTal22concentration}, but tight bias bounds remain unclear in the settings of this paper.
    \item \underline{Broader use of these techniques?} The techniques we use in this paper apply at the generality of Contractive Noisy Iterations, which is broader than just the Langevin Algorithm. Can these techniques shed light on the rapid mixing of other sampling algorithms? Most first-order algorithms can be viewed, at least intuitively, as variants of the Langevin Algorithm (see the discussion in \S\ref{ssec:intro-prev}). In particular, can our techniques help understand the mixing time of more sophisticated sampling algorithms which can perform better in practice, such as the Metropolis-Adjusted Langevin Algorithm or Hamiltonian Monte Carlo?
    \item \underline{Reconciliation of mixing analyses?} As discussed in \S\ref{ssec:intro-prev}, our mixing analysis exploits connections between sampling over $\R^d$ and optimization over $\R^d$, whereas many recent analyses exploit connections between sampling over $\R^d$ and optimization over the space $\cP(\R^d)$ of probability distributions. From both a technical as well as pedagogical perspective, it is interesting to ask: to what extent can these seemingly disparate techniques be reconciled? 
    \item \underline{Beyond log-concavity?} An important but difficult question for bridging theory and practice is to understand to what extent fast mixing carries over to families of target distributions $\pi$ under weaker assumptions than log-concavity (i.e., weaker assumptions on potentials than convexity). There have been several recent breakthrough results towards this direction under various assumptions like dissipativity or isoperimetry or under weaker notions of convergence (see the previous work section \S\ref{ssec:intro-prev} and the very recent papers~\citep{balasubramanian2022towards,chewi2022fisher}). However, on one hand tight mixing bounds are unknown for any of these settings (see Footnote~\ref{fn:fisher}), and on the other hand the non-convexity assumptions made in the sampling literature are often markedly different from those in the optimization literature. To what extent can either of these issues be reconciled? Can our techniques help shed light on either question? The key issue with extending our techniques to these settings is that without convexity of the potential, a Gradient Descent step is not contractive, which seems to preclude using the Privacy Amplification by Iteration argument. 
\end{itemize}

	\paragraph*{Acknowledgements.} We are grateful to Adam Smith for an insightful discussion about~\citep{AltTal22dp} which led us to the realization of the rapid mixing in Theorem~\ref{thm:convex-ub}, to John Duchi for a reference to Lemma~\ref{lem:tv-worst}, to Sinho Chewi for sharing an early draft of his book~\citep{chewibook}, and to Murat Erdogdu and Andre Wibisono for encouragement and helpful conversations about the related literature.
	
	\small
	\appendix
	
\section{Deferred Details}\label{app:deferred}

\subsection{Contractivity of Gradient Steps}\label{app:contractivity}

\begin{proof}[Proof of Lemma~\ref{lem:contractive-gd}]
	Since our purpose here is to explain the idea behind this classical lemma, we will assume for simplicity that $f$ is twice differentiable. (A proof for the general case can be found in standard convex optimization texts.)
	Consider any two points $x,x'$. By the Fundamental Theorem of Calculus,
	\[
	\nabla f(x) - \nabla f(x')
	=
	\int_0^1 \nabla^2 f \left( tx + (1-t)x' \right) (x - x') dt \,.
	\]
	Thus
	\begin{align*}
		\left\| \left(x - \eta \nabla f(x)\right) - \left(x' - \eta \nabla f(x')\right) \right\|
		&=
		\left\| \left( I_d - \eta \int_0^1 \nabla^2 f \left( tx + (1-t)x' \right)  dt \right)(x - x') \right\|
		\\ &\leq
		\left\| I_d - \eta \int_0^1 \nabla^2 f \left( tx + (1-t)x' \right) dt \right\| \left\| x - x' \right\|
		\\ &\leq \max_{\lambda \in [m,M]} \abs{ 1 - \eta \lambda } \left\| x - x' \right\|
				\\ &= c \left\| x - x' \right\|\,.
	\end{align*}
	Above, the first step is by factoring, the second step is by sub-multiplicativity of the operator norm, the third step is because a twice differentiable function being $m$-strong convexity and $M$-smoothness is equivalent to its Hessians having eigenvalues within $[m,M]$, and the last step is because $\lambda \mapsto \abs{1 - \eta \lambda}$ is convex and thus maximized at the boundary of the interval $[m,M]$. This establishes $c$-contractivity, as desired.
\end{proof}

\subsection{Privacy Amplification by Iteration}\label{app:pabi}

Here we recall (simplified forms of) the two lemmas underpinning the PABI technique, their proofs, and how they imply Proposition~\ref{prop:pabi-new}. These proofs have already appeared in previous work on PABI~\citep{pabi,AltTal22dp} and are provided here only for completeness and for the sake of accessibility of this article beyond the differential privacy community. See the techniques section \S\ref{ssec:intro-tech} for a high-level overview of the PABI technique. 

\subsubsection{Contraction-reduction lemma}

We recall this lemma in a simplified form, where both measures are pushed-forward through the \emph{same} contraction, as this suffices for the present purpose of sampling analyses (in differential privacy one must consider two different contraction maps). In this setting, the proof simplifies, and the lemma can be interpreted as an extension of the classical data-processing inequality for R\'enyi divergence to shifted R\'enyi divergence---under the condition that the data-processing operation is contractive. 

\begin{lemma}[Contraction-reduction lemma, simplified version]\label{lem:cr}
	For any distributions $\mu,\nu \in \cP(\R^d)$ and (random) $c$-contractive function $\phi : \R^d \to \R$,
	\[
	\Dalpshiftlr{cz}{\phi_{\#} \mu}{\phi_{\#} \nu}
	\leq
	\Dalpshiftlr{z}{\mu}{\nu}\,.
	\]
\end{lemma}
\begin{proof}
	Let $\mu'$ be a distribution verifying $\Dalpshift{z}{\mu}{\nu}$; that is, $\mu'$ satisfies $W_{\infty}(\mu,\mu') \leq z$ and $\Dalp{\mu'}{\nu} = \Dalpshift{z}{\mu}{\nu}$. Because $\phi$ is a $c$-contraction, it follows that $W_{\infty}(\phi_{\#}\mu,\phi_{\#}\mu') \leq cW_{\infty}(\mu,\mu') \leq cz$. Thus 
	\[
		\Dalpshiftlr{cz}{\phi_{\#} \mu}{\phi_{\#} \nu} \leq \Dalplr{\phi_{\#}\mu'}{\phi_{\#}\nu} \leq \Dalplr{\mu'}{\nu} = \Dalpshift{z}{\mu}{\nu}\,,
	\]
	where the middle step is due to the data-processing inequality for R\'enyi divergences (Lemma~\ref{lem:renyi-dataprocess}).
\end{proof}

\subsubsection{Shift-Reduction Lemma}

We state this lemma in the form used in differential privacy, with only the minor simplification that the noise does not come from an arbitrary distribution on a Banach space, but instead is simply a Gaussian distribution on $\R^d$. See the techniques section \S\ref{ssec:intro-tech} for intuition of how this lemma can be thought of as a strengthening of the trivial bound
\begin{align}
	\Dalplr{\mu \ast \xi}{\nu \ast \xi} 
	\leq
	\Dalplr{\mu \otimes \xi}{\nu \otimes \xi}
	= 
	\Dalplr{\mu}{\nu} + \Dalplr{\xi}{\xi}
	\label{eq:sr-intuition}
\end{align}
which follows from the data-processing and tensorization properties of R\'enyi divergence---where in words, the strengthening is achieved by moving $a$ units of ``displacement'' between $\mu$ and $\nu$ from the first term to the second term.

\begin{lemma}[Shift-reduction lemma]\label{lem:sr}
	For any distributions $\mu,\nu \in \cP(\R^d)$ and shift $a \geq 0$,
	\[
	\Dalpshiftlr{z}{\mu \ast \cN(0,\sig^2 I_d)}{ \nu \ast \cN(0,\sig^2 I_d)}
	\leq
	\Dalpshiftlr{z+a}{\mu}{\nu}
	+
	\frac{\alpha a^2}{2\sig^2}\,.
	\]
\end{lemma}

Below we reproduce the proof of this lemma from~\citep{pabi} for completeness, although with a slightly modified presentation that emphasizes the connection to~\eqref{eq:sr-intuition}, as isolated in the following observation. Intuitively, this lemma can be thought of as a variant of the Chain Rule for the KL divergence, with expectation replaced by supremum so that the inequality extends from the KL divergence ($\alpha=1$) any R\'enyi divergence ($\alpha \geq 1$).

\begin{lemma}[Analog of~\eqref{eq:sr-intuition} used in Lemma~\ref{lem:sr}]
	\label{lem:renyi-tensor-dependent}
	Consider any random vectors $X,Y,X',Y'$ in $\R^d$. If $X'$ and $Y'$ are independent, then
	\[
		\Dalplr{X + Y}{X' + Y'}
		\leq
		\Dalplr{X}{X'}
		+
		\esssup_{x} \Dalplr{Y \big|_{X=x}} {Y'}
	\]
	where $Y\big|_{X=x}$ denotes the distribution of $Y$ conditional on $X=x$.
\end{lemma}
\begin{proof}
	By the data-processing property of the R\'enyi divergence, the definition of the R\'enyi divergence, the independence assumption, H\"older's inequality, and again the definition of R\'enyi divergence,
	\begin{align*}
		\Dalplr{X + Y}{X' + Y'}
		&\leq \Dalplr{(X,Y)}{(X',Y')}
		\\ &=
		\frac{1}{\alpha - 1} \log \iint \left( \frac{\Prob_{(X,Y)}(x,y) }{ \Prob_{(X',Y')}(x,y)} \right)^{\alpha} \Prob_{(X',Y')}(x,y) dx dy
		\\ &=
		\frac{1}{\alpha - 1} \log \iint \left( \frac{ \Prob_X(x) \Prob_{Y|X=x}(y) }{ \Prob_{X'}(x) \Prob_{Y'}(y) } \right)^{\alpha} \Prob_{X'}(x) \Prob_{Y'}(y) dx dy
		\\ &\leq
		\frac{1}{\alpha - 1} \log \int \left(  \frac{\Prob_{X}(x)}{\Prob_{X'}(x)} \right)^{\alpha}  \Prob_{X'}(x) dx \cdot \esssup_x \int \left(  \frac{\Prob_{Y|X=x}(y)}{\Prob_{Y'}(y)} \right)^{\alpha}  \Prob_{Y'}(y) dy
		\\ &=
		 \Dalplr{X}{X'}
		+
		\esssup_{x} \Dalplr{Y \big|_{X=x}} {Y'} 
	\end{align*}
\end{proof}

\begin{proof}[Proof of Lemma~\ref{lem:sr}]
	\underline{Case of $z=0$.} Let $\mu'$ be a distribution guaranteed by the definition of the shifted R\'enyi divergence $\Dalpshift{a}{\mu}{\nu}$; that is, $\mu'$ is a distribution satisfying $W_{\infty}(\mu,\mu') \leq a$ and $\Dalplr{\mu'}{\nu} = \Dalpshiftlr{a}{\mu}{\nu}$. Let $(U,U')$ be the coupling of $(\mu,\mu')$ guaranteed by the definition of $W_{\infty}(\mu,\mu')$; then $W := U - U'$ satisfies $\|W\| \leq a$ almost surely. 
	Now bound 
	\begin{align*}
		\Dalplr{\mu \ast \cN(0,\sig^2 I_d)}{ \nu \ast \cN(0,\sig^2 I_d)}
		&= 
		\Dalplr{\mu' \ast \cN(W,\sig^2 I_d)}{ \nu \ast \cN(0,\sig^2 I_d)}
		\\ &\leq 
		\Dalplr{\mu'}{\nu} + \sup_{w \; : \; \|w\| \leq  a} \Dalplr{\cN(w,\sig^2 I_d)}{ \cN(0,\sig^2 I_d)}
		\\ &=
		\Dalpshiftlr{a}{\mu}{\nu} + \frac{\alpha a^2}{2\sig^2} \,.
	\end{align*}
	Above, the first step is by adding and subtracting $W$, 
	the second step is by Lemma~\ref{lem:renyi-tensor-dependent}, and the third step is by construction of $\mu'$ and the
	identity $\Dalplr{\cN(w,\sig^2 I_d)}{\cN(0,\sig^2 I_d)} = \frac{\alpha \|w\|^2}{2\sig^2}$ which can be found in, e.g.,~\citep[Example 3]{van2014renyi}.

	\underline{Case of general $z \geq 0$.} Define $\mu', U, U', W$ analogously, now in terms of $\Dalpshift{z+a}{\mu}{\nu}$ rather than $\Dalpshift{a}{\mu}{\nu}$. Decompose $W = W_1  + W_2$ for $W_1 = W \cdot \min(1,z/\|W\|)$; this ensures $\|W_1\| \leq z$ and $\|W_2\| \leq a$. Observe that
	\begin{align*}
		\Dalpshiftlr{z}{\mu \ast \cN(0,\sig^2 I_d)}{\nu \ast \cN(0,\sig^2 I_d)}
		\leq
		\Dalplr{\mu \ast \Prob_{W_1} \ast \cN(0,\sig^2 I_d)}{\nu \ast \cN(0,\sig^2 I_d)}
		\leq
		\Dalpshiftlr{a}{\mu \ast \Prob_{W_1}}{\nu} + \frac{\alpha a^2}{2\sig^2}\,.
	\end{align*}
	Above, the first step is because $\|W_1\| \leq z$, 
	and the second step is by the case $z=0$ proved above. The proof is complete by observing that
	\begin{align*}
		\Dalpshiftlr{a}{\mu \ast \Prob_{W_1}}{\nu}
		\leq
		\Dalplr{\mu \ast \Prob_{W_1} \ast \Prob_{W_2}}{\nu}
		= 
		\Dalplr{\mu'}{\nu}
		= 
		\Dalpshiftlr{z+a}{\mu}{\nu}\,,
	\end{align*}
	where above the first step is by the bound $\|W_2\| \leq a$, the second step is by the decompositions $W = W_1 + W_2$ and $U' = U + W$, and the third step is by the bound $\|W\| \leq z+a$.
\end{proof}

\subsubsection{Proof of Diameter-Aware PABI Bound}

Finally, we recall how the combination of Lemmas~\ref{lem:cr} and~\ref{lem:sr} imply the diameter-aware PABI bound in Proposition~\ref{prop:pabi-new}. This argument is from~\citep{AltTal22dp} and is reproduced below for completeness, albeit here with the simplification that here the two CNI use the same contractions, since this suffices for sampling analyses.

\begin{proof}[Proof of Proposition~\ref{prop:pabi-new}]
	Consider any sequence $a_1, \dots, a_T \geq 0$ satisfying $D = \sum_{t=1}^T c^{-t} a_i$. Relate the shifted R\'enyi divergence at time $T$ to that at time $T-1$ by bounding
	\begin{align*}
		\Dalplr{X_T}{X_T'}
		&=
		\Dalpshiftlr{0}{\phi_T(X_{T-1})+Z_{T-1}}{\phi_T(X_{T-1}')+Z_{T-1}}
		\\ &\leq
		\Dalpshiftlr{a_T}{\phi_T(X_{T-1})}{\phi_T(X_{T-1}')} + \frac{\alpha a_T^2}{4\eta}
		\\ &\leq
		\Dalpshiftlr{ca_T}{X_{T-1}}{X_{T-1}'} + \frac{\alpha a_T^2}{2\sig^2}\,,
	\end{align*}
	where above, the first step is by definition of the CNI and the fact that the standard R\'enyi divergence is the same as $0$-shifted R\'enyi divergence; the second step is by the shift-reduction lemma (Lemma~\ref{lem:sr}); the third step is by the contraction-reduction lemma (Lemma~\ref{lem:cr}).
	\par By repeating this argument $T$ times, we can relate the shifted R\'enyi divergence at time $T$ to that at time $0$ by bounding
	\begin{align*}
		\Dalplr{X_T}{X_T'}
		\leq \Dalpshiftlr{\sum_{t=1}^T c^{-i} a_i }{X_0}{X_0'} + \frac{\alpha}{2 \sig^2}\sum_{t=1}^T a_T^2\,.
	\end{align*}
	Now this shifted R\'enyi divergence at time $0$ vanishes because $X_0,X_0'$ both lie in the set $\cK$ which has diameter $D = \sum_{t=1}^T c^{-t} a_t$. We conclude that
	\begin{align}
		\Dalplr{X_T}{X_T'}
		\leq \frac{\alpha}{2 \sig^2} \cdot \Bigg(\;\inf_{\substack{a_1, \dots, a_T \geq 0 \\ \text{s.t. } \sum_{t=1}^T c^{-t} a_t = D}} \sum_{t=1}^T a_T^2 \;\Bigg) \,.
		\label{eq:pabi-proof}
	\end{align}
	The proof is complete by setting $a_1 = \dots = a_t = D/T$ if $c = 1$, or setting $a_t = c^{-T}D \mathds{1}_{t  = T}$ if $c < 1$.
\end{proof}

\begin{remark}[PABI bound that is continuous at $c=1$]\label{rem:pabi-continuous}
	The statement of Proposition~\ref{prop:pabi-new} is discontinuous in $c$, as mentioned in the discussion preceding it. This is just for simplicity of the exposition. Indeed, by actually solving the optimization problem in~\eqref{eq:pabi-proof}, we obtain a slightly sharper bound that is continuous in $c$. Namely, by taking the optimal solution is $a_t = c^{-t} \beta D$ where $\beta = (c^2 - 1)/(1-c^{-2T})$, we have
		\begin{align}
		\Dalplr{X_T}{X_T'}
		\leq \frac{\alpha D^2}{2 \sig^2} \cdot \frac{c^2 - 1}{1-c^{-2T}}\,.
		\label{eq:pabi-bound}
	\end{align}
	This is continuous in $c$ because $\lim_{c \to 1} (c^2- 1)/(1-c^{-2T}) = 1/T$. However, we keep Proposition~\ref{prop:pabi-new} as stated because the exact bound~\eqref{eq:pabi-bound} is harder to parse and only yields a small constant improvement for $c < 1$.
\end{remark}

\subsection{Mixing in Other Notions of Distance}\label{app:otherdists}

\subsubsection{Preliminaries on Dobrushin's Coefficient and $f$-Divergences}\label{ssec:prelim:mixing}

	A simple, helpful lemma for proving fast mixing bounds for the total variation metric is the following. Operationally, it reduces proving mixing bounds to arbitrary error $\eps$ to just the case of constant $\eps$, say $1/4$. This holds for arbitrary Markov chains (not just the Langevin Algorithm), and is a straightforward consequence of the submultiplicativity of the function $t \mapsto \sup_{\mu,\nu} \|P^t \mu - P^t \nu\|_{\TV}$ for any Markov kernel $P$; a proof can be found in, e.g.,~\citep[\S 4.5]{peres2017}.
	
	\begin{lemma}[Mixing times are exponentially fast for total variation]\label{lem:tv-mixing-exponential}
		\[
		T_{\mathrm{mix},\,\TV}(\eps) \leq \lceil \log_2 \eps^{-1} \rceil \cdot T_{\mathrm{mix},\,\TV}(1/4) \,.
		\]
	\end{lemma}
	
	\par Although Lemma~\ref{lem:tv-mixing-exponential} and its proof are tailored to the total variation metric, this fact is helpful for proving rapid mixing bounds for other divergences. This is due to the following lemma from information theory which states that, among all $f$-divergences, the ``mixing rate'' is worst for the total variation distance. This mixing rate is quantified in terms of the so-called strong data processing constant $\rho_{\cD}(P) := \sup_{\mu \neq \nu} \cD (P\mu \| P\nu) / \cD (\mu \| \nu)$ for a Markov chain $P$ with respect to an $f$-divergence $\cD$, a.k.a., the Dobrushin contraction coefficient in the case of the total variation distance $\cD = \TV$.
	For further details on this lemma, see~\citep[\S11.1]{duchinotes}.

	\begin{lemma}[Total variation mixing is worst case among all $f$-divergences]\label{lem:tv-worst} For any $f$-divergence $\cD$ and any Markov chain $P$,
		\[
		\rho_{\cD}(P) \leq \rho_{\TV}(P)\,,
		\]
	\end{lemma}
	
	Lemma~\ref{lem:tv-worst} applies, for example, to the KL, Chi-Squared, and Hellinger divergences since those are all $f$-divergences. However, applying this lemma to the R\'enyi divergence requires an intermediate step since the R\'enyi divergence is not an $f$-divergence. This can be achieved by relating the R\'enyi divergence $\cD_{\alpha}$ to the generalized Hellinger divergence $\Hell_{\alpha}$, and recalling that the latter is an $f$-divergence. We record the necessary preliminaries for this connection in the following lemma; further details can be found in, e.g.,~\citep[equations (52) and (80)]{sason2016f}.
	
	\begin{lemma}[Preliminaries for using Lemma~\ref{lem:tv-worst} for R\'enyi divergence]\label{lem:tv-worst-renyi}
		Suppose $\alpha \in (1,\infty)$. Then $\Hell_{\alpha}$ is an $f$-divergence and satifies
		\[
		\cD_{\alpha}(\mu,\nu) = \frac{1}{\alpha-1}\log \left(  1 + (\alpha - 1) \Hell_{\alpha}(\mu,\nu) \right) 
		\,.
		\]
	\end{lemma}

\subsubsection{Proof of Corollary~\ref{cor:convex-mixing}}
	By a basic mixing time argument (Lemma~\ref{lem:tv-mixing-exponential}) which applies to any Markov chain, we can boost the total variation mixing time proved in Theorem~\ref{thm:convex-ub} for constant error $1/4$ to arbitrarily small error $\eps > 0$, namely
	\[
	T_{\mathrm{mix},\,\TV}(\eps) \leq \frac{2D^2}{\eta} \lceil \log_2(1/\eps) \rceil
	\,.
	\]

	\par We now show how this fast mixing bound for total variation implies fast mixing for the R\'enyi divergence $\cD_{\alpha}$ for $\alpha \in [1,\infty)$. 
	We do this by
	proving fast mixing for the more stringent notion of the Hellinger divergence $\Hell_{\alpha}$. Our analysis proceeds in two phases. First, we argue about the number of iterations it takes to mix to error $2e^{\alpha - 1}$ in Hellinger divergence. To this end, observe that $T_{\mathrm{mix},\,D_{\alpha}}(1) \lesssim \alpha D^2/\eta$
	by~\eqref{eq:convex-proof-constant}; and also observe that if the R\'enyi divergence $\cD_{\alpha}$ between two distributions is at most $1$, then the Hellinger divergence $\Hell_{\alpha}$ is at most $ 2e^{\alpha - 1}$ due to the identity in Lemma~\ref{lem:tv-worst-renyi} combined with a crude bound that uses the inequality $e^{\alpha - 1} \leq 1+2(\alpha - 1)$ for $\alpha \in [1,2]$. Together, these two observations imply
	\[
	T_{\mathrm{mix},\Hell_{\alpha}}\left( 2e^{\alpha - 1} \right)
	\leq
	T_{\mathrm{mix},\,D_{\alpha}}(1)
	\lesssim 
	\frac{\alpha D^2}{\eta}\,.
	\]
	Next, we claim that it takes at most
	\[
	\frac{D^2}{\eta} \log  \frac{2e^{\alpha - 1}}{\eps}\lesssim \frac{D^2}{\eta} \left[ (\alpha - 1) + \log \frac{1}{\eps} \right]\,.
	\]
	iterations to mix to $\eps$ error in Hellinger divergence $\Hell_{\alpha}$ after we have reached error $2e^{\alpha - 1}$. This is a consequence of two observations: (i) the Dobrushin contraction coefficient corresponding to $D^2/\eta$ iterations of the Langevin Algorithm is bounded above by $1/4$ by~\eqref{eq:convex-proof-TV-partial}; and (ii) the analogous contraction coefficient for the Hellinger divergence $\Hell_{\alpha}$ is no worse by Lemma~\ref{lem:tv-worst} and the fact that $\Hell_{\alpha}$ is an $f$-divergence (Lemma~\ref{lem:tv-worst-renyi}). By combining the above two displays, we conclude that
	\[
	T_{\mathrm{mix},\,\Hell_{\alpha}}(\eps)
	\lesssim
	\frac{D^2}{\eta} \left[ (\alpha-1) + \log \frac{1}{\eps} \right]\,.
	\]
	Now, because the R\'enyi divergence $\cD_{\alpha}$ is always bounded above by the Hellinger divergence $\Hell_{\alpha}$ (this follows from Lemma~\ref{lem:tv-worst-renyi} and the elementary inequality $\log(1+x) \leq x$), the above display implies the claimed mixing time bound for the R\'enyi divergence. This immediately implies the claimed fast mixing bounds in the other divergences (KL, Hellinger, and Chi-Squared) due to the standard comparison inequalities in Lemma~\ref{lem:renyi-ineq}.

\subsection{Extension to Unconstrained Settings}\label{app:diam}

Here, we briefly mention how our analysis technique readily extends unchanged to unconstrained settings, i.e., if $\cK = \R^d$, or equivalently the Langevin Algorithm makes no projection in its update. This leads to tight mixing bounds and only requires two simple changes in the analysis.

\par First, we must slightly restrict the initialization, since otherwise the Langevin Algorithm takes arbitrarily long to even come close to where $\pieta$ is concentrated. To remedy this, simply initialize at any mode $x^*$ of $\pi$, a.k.a., any minimizer of the potential $f$. Such an $x^*$ is readily computed using Gradient Descent; note that this uses the same first-order access to gradients of $f$ as the Langevin Algorithm. 

\par Second, the diameter $D$ is replaced by a proxy $D_{\eta,\eps}$ which is the diameter of a ball that captures all but $\eps$ mass of $\pieta$. Formally, let $D_{\eta,\eps}$ satisfy
\begin{align}
	\Prob_{X \sim \pieta} \left[ \|X - x^*\| \geq D_{\eta,\eps} \right] \leq \eps\,.\label{eq:diam-Dproxy}
\end{align}
for any minimizer $x^*$ of $f$. Bounds on $D_{\eta,\eps}$ for both the convex and strongly convex settings are in~\citep{AltTal22concentration}. A bound for the strongly convex setting can also be obtained by Theorem 8 of the recent revision of~\citep{VempalaW19}.


\par With these two minor changes, the proofs of our $\TV$ mixing time upper bounds for convex potentials (Theorem~\ref{thm:convex-ub}) and strongly convex potentials (Theorem~\ref{thm:sc-ub}) carry over to the present unconstrained setting, establishing that for the same mixing time $T$ (except with $D$ replaced by $D_{\eta,\eps}$), we have $\TV(X_T, \pi_{\eta,\eps}) \leq \eps$, where $\pi_{\eta,\eps}$ is the truncation of $\pi_{\eta}$ to the ball of radius $D_{\eta,\eps}$; that is, $\pi_{\eta,\eps}(x) \propto \pieta(x) \cdot \mathds{1}[\|x\| \leq D_{\eta,\eps}]$. By definition of $\pi_{\eta,\eps}$, we have $\TV(\pieta,\pi_{\eta,\eps}) \leq 2 \eps$. Thus $\TV(X_T, \pieta) \leq 3\eps$ by the triangle inequality, so we conclude that in this unconstrained setting, $T$ iterations suffices for the Langevin Algorithm initialized at $x^*$ to converge to $\TV$ distance $3\eps$ from $\pieta$.

\par This leads to tight mixing bounds. Indeed, for the case of strongly convex potentials, our mixing lower bound matches up to a mild logarithmic term, see the discussion in \S\ref{sec:sc}. And for the case of convex potentials, the reachability lower bound argument in \S\ref{ssec:convex-lb} applies to the unconstrained setting here: it takes roughly $D_{\eta,\eps}^2/\eta$ iterations for the Langevin Algorithm to move a distance of $D_{\eta,\eps}$ and thus to reach all parts of the ball of diameter $D_{\eta,\eps}$, and this is a pre-requisite for mixing to $\Theta(\eps)$ error in $\TV$.

	\footnotesize
	\addcontentsline{toc}{section}{References}
	\bibliographystyle{plainnat}
	\bibliography{sampling_pabi}{}

\end{document}